\documentclass[12 pt]{amsart}%
\usepackage{hyperref, graphicx, mathrsfs}
\usepackage{amssymb,amsmath,amsthm}
\usepackage{mathrsfs,dsfont}
\usepackage{amsmath}
\usepackage{amsfonts}
\usepackage{amssymb}
\usepackage{graphicx}%
\providecommand{\U}[1]{\protect\rule{.1in}{.1in}}

\theoremstyle{plain}
\newtheorem{theorem}{Theorem}[section]

\newtheorem{proposition}{Proposition}[section]
\newtheorem{lemma}[proposition]{Lemma}

\newtheorem{corollary}{Corollary}

\newtheorem*{remark}{Remark}
\numberwithin{equation}{section}
\DeclareMathSymbol{\Gamma}{\mathalpha}{letters}{"00}
\begin{document}
\title[Unstable Manifolds]{Unstable manifolds of Euler equations}
\author[Lin]{Zhiwu Lin$^{\dagger}$}
\address{$^{\dagger}$ School of Mathematics\\
Georgia Institute of Technology\\
Atlanta, GA 30332}
\email{zlin@math.gatech.edu}
\author[Zeng]{Chongchun Zeng$^{*}$}
\address{$^{*}$School of Mathematics\\
Georgia Institute of Technology\\
Atlanta, GA 30332}
\email{zengch@math.gatech.edu}
\thanks{$^{\dagger}$ The first author is funded in part by NSF DMS 0908175.}
\thanks{$^{*}$ The second author is funded in part by NSF DMS 0801319.}
\date{}

\begin{abstract}
We consider a steady state $v_{0}$ of the Euler equation in a fixed bounded
domain in $\mathbf{R}^{n}$. Suppose the linearized Euler equation has an
exponential dichotomy of unstable and center-stable subspaces. By rewriting
the Euler equation as an ODE on an infinite dimensional manifold of volume
preserving maps in $W^{k, q}$, $(k>1+\frac{n}{q})$, the unstable (and stable)
manifolds of $v_{0}$ are constructed under certain spectral gap condition
which is verified for both $2$D and $3$D examples. In particular, when the
unstable subspace is finite dimensional, this implies the nonlinear
instability of $v_{0}$ in the sense that arbitrarily small $W^{k, q}$
perturbations can lead to $L^{2}$ growth of the nonlinear solutions.

\end{abstract}
\maketitle

\section{Introduction}

\label{S:intro}

We consider the incompressible Euler equation on a smooth bounded domain
$\Omega\subset\subset\mathbb{R}^{n}$, $n \ge2$, under the slip (or periodic in
certain directions) boundary condition
\[%
\begin{cases}
v_{t} + (v\cdot\nabla) v = -\nabla p \quad\text{ and } \quad\nabla\cdot v =0
\qquad & x\in\Omega\\
v \cdot N =0, \qquad & x\in\partial\Omega
\end{cases}
\tag{E}
\]
where $v= (v^{1}, \ldots, v^{n})^{T}$ is the velocity field and $N$ is the
unit outward normal vector of $\Omega$. We take
\begin{equation}
\label{E:phaseS}W_{Euler}^{k, q} \triangleq\{v \in W^{k, q} (\Omega,
\mathbb{R}^{n}) \mid\nabla\cdot v=0 \text{ in } \Omega, \, v\cdot N = 0 \text{
on } \partial\Omega\}, \qquad q>1, \quad k> 1+\frac nq
\end{equation}
as the phase space. It is well known that (E) is well posed in these spaces,
globally if $n=2$ and locally if $n\ge3$. As shown below, the pressure $p$ can
be written in terms of $v$ through a quadratic mapping.

Let $v_{0}$ be a steady solution of (E). Linearize the equation at (E) and we
obtain
\begin{equation}
v_{t}=-(v_{0}\cdot\nabla)v-(v\cdot\nabla)v_{0}-\nabla p\triangleq
Lv\label{E:LEuler}%
\end{equation}
where the operator $L$ can be defined as acting only on $v$ since the
linearized pressure $p$ can be determined by $v$ linearly, though non-locally.
To study the dynamics near $v_{0}$, the first step is to understand linear
instability, that is, the spectrum of the operator $L$. The problem of linear
instability of inviscid flows has a long history dated back to Rayleigh and
Kelvin in 19th century. But even until now, very few sufficient conditions for
the existence of unstable eigenvalues are known and most of the investigations
had been restricted to shear flows and rotating flows. See \cite{dr81}%
,\cite{dh66} and the references therein. Some recent results on instability
conditions can be found in \cite{Lin03} \cite{lin-comt} for shear flows and
rotating flows, and in \cite{lin-cmp} for general 2D flows. Besides the
discrete unstable spectrum, the linearized Euler operator may also have
non-empty unstable essential spectrum due to nontrivial Lyapunov exponents of
the steady flow $v_{0}$ (\cite{fv-91} \cite{lif-mai}
\cite{latushkin-shvydkoy-jfa} \cite{la-sv-04} \cite{vishik96}). Indeed, growth
of linearized solutions can be seen in $H^{s}~\left(  s>1\right)  $ norm near
any nontrivial steady flows, due to the stretching of the steady fluid
trajectory. One also notes that the choice of the Sobolev space (norm)
actually affects the essential spectrum which corresponds to small spatial
scales, but not discrete spectrum corresponding to large scales.

Consider a linearly unstable steady flow $v_{0}$, that is, the linearized
Euler operator $L$ has an unstable discrete eigenvalue. To discuss the
nonlinear instability, it is important to specify the norms. On the one hand,
certain regularity is necessary in the local well-posedness of classical
solutions. On the other hand, as mentioned in the above, the choice of the
norm already affects the essential spectrum at the linear level. Moreover,
even near steady states without unstable eigenvalues, solutions are expected
to grow in $H^{s}$ norm with $s>1$. Therefore the growth in the energy norm
$L^{2}$ of nonlinear solutions is more a nonlinear reflection of the linear
instability from the discrete spectrum, which also corresponds to the
instability in the large scale spatial scale (see \cite[Section 6.2]{Lin04}
for detailed discussions). Naturally, the ideal nonlinear instability result
would be to obtain order $O(1)$ growth in $L^{2}$ distance (weaker energy
norm) from the steady state $v_{0}$ of solutions starting with arbitrary small
initial perturbation from $v_{0}$ in $H^{s}\ $norm (stronger norm), where
$s>0\ $is determined by the regularity of unstable eigenfunctions. Such
nonlinear instability result (i.e. $H^{s}\rightarrow L^{2}$)\ is not only
mathematically stronger, but also physically interesting due to the
discussions in the above.

The proof of the nonlinear instability based on unstable eigenvalues is
nontrivial for several reasons. The main difficulty is that the nonlinear term
$v\cdot\nabla v$ contains a loss of derivative. Moreover, the norm-dependent
unstable essential spectrum corresponds to growth in small spatial scales. It
may interact with discrete unstable modes and then cause complications in
proving nonlinear instability. In the last decade, there appeared several
proofs of nonlinear instability for Euler equations (\cite{bgs} \cite{gre00}
\cite{vf03} \cite{fsv} \cite{Lin04}). In \cite{fsv}, nonlinear instability in
$H^{s}\ \left(  s>1+\frac{n}{2}\right)  $ norm was proven for $n-$dimensional
Euler equations, under a spectral gap condition which was verified for $2$D
shear flows. All other papers proved nonlinear instability only for the $2$D
case, in the more desirable $H^{1}$ or $L^{2}$ norms and under different
spectral assumptions. Particularly, in \cite{Lin04} nonlinear instability in
$L^{2}$ norm was proved for general linearly unstable flows of $2$D Euler,
without additional assumption on the growth rate which was made in (\cite{bgs}
\cite{gre00} \cite{vf03}).

When there exist a collection $\sigma_{u}$ of unstable eigenvalues of the
linearized Euler operator at the steady flow $v_{0}\ $with strictly larger
real parts than the rest of the spectrum, it is very natural from a dynamical
system point of view to ask if a locally invariant unstable manifold tangent
to the eigenspace of $\sigma_{u}$ exists. The answer to this question would
provide a better picture of the nonlinear instability, including the
dimensions and the directions of the unstable solutions with certain minimal
growth rate. More importantly, such locally invariant manifolds provide more
precise characterization of the local dynamical pictures near an unstable
equilibria, and are also basic tools for constructing globally invariant
structures such as heteroclinic and homoclinic orbits. These dynamical
structures\textit{\ }are important in understanding the turbulent fluid
behaviors. The major obstacle to the construction of local invariant manifolds
is again the loss of derivative due to the derivative nonlinearity
$v\cdot\nabla v$. For dissipative models such as reaction-diffusion equations
(\cite{henry}) and Navier-Stokes equations (\cite{yudovich-book}
\cite{Li-NS-limit}), it is rather standard to construct invariant manifolds
since the dissipation terms provide strong smoothing effect to overcome the
loss of derivatives in the nonlinear terms. However, for non-dissipative
continuum models including Euler equations, the linearized operators have no
smoothing effect to help overcome the loss of derivative in the nonlinear
terms. So it has been largely open to construct invariant manifolds for
conservative continuum models such as Euler equations. In the proof of
nonlinear instability for 2D Euler, one takes initial data perturbed along the
direction of unstable eigenfunctions and then uses special properties of
nonlinear solutions of 2D Euler to overcome the loss of derivatives, such as
the bootstrap arguments in (\cite{bgs} \cite{Lin04} \cite{vf03}) and the
nonlinear energy estimates in (\cite{gre00}). However, these techniques can
not be used for constructing invariant manifolds, since we do not know
beforehand the initial conditions for solutions on unstable or stable manifolds.

In this paper, we obtain the first result on stable and unstable manifolds of
Euler equations in any dimensions. To state the precise results, we formulate
the following assumptions:

\begin{enumerate}
\item[(A1)] $v_{0}\in W_{Euler}^{k+r,q}$, where
$r\geq4$.

\item[(A2)] $\exists\; \lambda_{u} > \lambda_{cs} >0$, and closed subspaces
$X_{u}$ and $X_{cs}$ of $W_{Euler}^{k, q}$ such that they satisfy $L
\big(X_{u, cs} \cap domain(L)\big) \subset X_{u, cs}$ respectively and
$W_{Euler}^{k, q} = X_{u} \oplus X_{cs}$. Moreover, let $L_{u} = L|_{X_{u}}$
and $L_{cs} = L|_{X_{cs}}$, then for some $M>0$, they satisfy
\[
|e^{tL_{cs}}|\le M e^{\lambda_{cs}t}, \quad\forall\; t\ge0 \quad\text{ and }
\quad|e^{tL_{u}}|\le M e^{\lambda_{u} t}, \quad\forall\; t\le0 .
\]

\item[(A3)] The largest Lyapunov exponent $\mu_{0}$ (in both forward and back
time) of the linearized equations
\[
y_{t}=Dv_{0}\big(x(t)\big)y\qquad y_{t}=-\Big(Dv_{0}\big(x(t)\big)\Big)^{\ast
}y,
\]
along any integral curve $x(t)$ of $x_{t}=v_{0}(x)$, satisfies
\[
\lambda_{u}-\lambda_{cs}>K_{0}\mu_{0}%
\]
where $K_{0}$ is a constant depending only on $r$ and $k$.
\end{enumerate}


Now we give our main theorem

\begin{theorem}
\label{T:UnstableM} Under above assumptions (A1)-(A3), there exists a unique
$C^{r-3,1}$ local unstable manifold $W^{u}$ of $v_{0}$ in $W_{Euler}^{k,q}$
which satisfies

\begin{enumerate}
\item It is tangent to $X_{u}$ at $v_{0}$.

\item It can be written as the graph of a $C^{r-3,1}\ $mapping from a
neighborhood of $v_{0}$ in $X_{u}$ to $X_{cs}$.

\item It is locally invariant under the flow of the Euler equation (E), i.e.
solutions starting on $W^{u}$ can only leave $W^{u}$ through its boundary.

\item Solutions starting on $W^{u}$ converges to $v_{0}$ at the rate
$e^{\lambda t}$ as $t\rightarrow-\infty$ for any $\lambda< \lambda_{u} -K_{0}
\mu_{0}$.
\end{enumerate}

The same results hold for local stable manifold of $v_{0}$ as the Euler
equation (E) is time-reversible.
\end{theorem}


\begin{remark}
In Sections 3 and 4, the assumptions (A2)-(A3) are verified for linearly
unstable $2$D shear flows and rotating flows, as well as $3$D shear flows. For
these flows, $\mu_{0}=0$ and thus (A3) is automatically satisfied if an
unstable eigenvalue exists.
\end{remark}

\begin{remark}
Suppose $v_{0} \in W^{k_{1}+r, q_{1}} \cap W^{k_{2}+r, q_{2}}$ and the
invariant decompositions $W^{k_{i}, q_{i}} = X_{u}^{i} \oplus X_{cs}^{i}$
along with the same exponents $\lambda_{u, cs}$ satisfy (A2) -- (A3) for
$i=1,2$. One may construct the local unstable manifolds $W^{u}_{i} \subset
W^{k_{i}, q_{i}}$, $i=1,2$ from Theorem \ref{T:UnstableM}. Assume $X_{u}^{1}=
X_{u}^{2}$, we claim $W_{1}^{u}=W_{2}^{u}$ on an open neighborhood of $v_{0}$.
In fact, let $k=\max\{k_{1}, k_{2}\}$ and $q=\max\{q_{1}, q_{2}\}$ and $W^{u}$
be the unstable manifold of $v_{0}$ in $W^{k, q} \subset W^{k_{i}, q_{i}}$,
$i=1,2$. Clearly $W^{u} \subset W_{i}^{u}$, $i=1,2$, with the same tangent
spaces and thus the above claim follows.
\end{remark}

When $X_{u}$ is finite dimensional, which in particular is always true in $2$D
under assumption the $\lambda_{u}>\lambda_{cs}>\left(  k-1\right)  \mu_{0}$ as
proved in Section \ref{S:2d}, then the $W^{k,q}$ topology and $L^{2}$ topology
are equivalent on $W^{u}$. Then an immediate consequence of the above theorem
is the nonlinear instability in $L^{2}$ norm with initial data slightly
perturbed from $v_{0}$ in $W^{k,q}$ norm.

\begin{corollary}
\label{C:instability} Suppose (A1) --(A3) are satisfied and $X_{u}$ is finite
dimensional, then there exists $\delta>0$ such that there exist a solution
$v(t)$ such that $|v(0) - v_{0}|_{L^{2}} \ge\delta$ and $|v(t) - v_{0}|_{W^{k,
q}} \to0$ as $t \to-\infty$ exponentially.
\end{corollary}

This is a stronger statement than the usual exponential nonlinear instability
which by definition means that there exists $\delta>0$ such that for any
$\epsilon>0$, there exists a solution $v(t)$ satisfying both $|v(0)-v_{0}
|<\epsilon$ and $\sup_{0\leq t\leq O(-\log\epsilon)}|v(t)-v_{0}|\geq\delta$.
One notes that the perturbed solution $v(t)$ is allowed to depend on
$\epsilon$ and no condition is imposed on the asymptotic behavior of $v(t)$ as
$t\rightarrow-\infty$. In the contrast, any solution $v(t)$ in Corollary
\ref{C:instability} satisfies, in addition to the requirements in the
nonlinear instability definition for all $\epsilon>0$, that it starts at
$v_{0}$ when $t=-\infty$ and get out of the $\delta$-neighborhood of $v_{0}$
in $L^{2}$ norm.

In the previous works (see references above) on nonlinear instability of the
Euler equation, growing solutions have usually been found in the most unstable
direction of the linearized equation \eqref{E:LEuler} with roughly the maximal
exponential growth rate. The above unstable manifold theorem actually provides
solutions growing in other relatively weaker unstable directions. Though not
necessary, it is easier to see this when $X_{u}$ is finite dimensional. In
fact, on the finite dimensional locally invariant manifold $W^{u}$, the Euler
equation (E) becomes a smooth ODE and $v_{0}$ is a hyperbolic unstable node.
When $L|_{X_{u}}$ has eigenvalues with different real parts, one may split
$X_{u}$ into strongly unstable subspace $X_{uu}$ and weaker unstable subspace
$X_{wu}$ such that $X_{u}=X_{uu}\oplus X_{wu}$. The standard invariant
manifold theory implies there exist the locally invariant weakly unstable
manifold $W^{wu}$ tangent to $X_{wu}$ and the locally invariant strongly
unstable fibers $W_{v}^{su}$ with base point $v\in W^{wu}$ and extend in the
direction of $X_{uu}$. Those solutions in $W^{wu}$ grow in the directions of
$X_{wu}$ at a slower exponential rate. Moreover, the Hartman-Grobman theorem
implies that a H\"{o}lder homeomorphism on $X^{u}$ may transform the Euler
equation restricted on $W^{u}$ into a linear ODE system.

In Section 4, we construct linearly unstable $3$D steady flow satisfying the
assumptions in Theorem \ref{T:UnstableM}. By Corollary \ref{C:instability},
this implies nonlinear exponential instability in $L^{2}$ norm. To our
knowledge, this is the first proof of nonlinear instability of $3$D Euler
equation. We note that the methods for proving nonlinear instability of $2$D
Euler cannot be applied to prove nonlinear instability for $3$D Euler. For
example, the bootstrap arguments in (\cite{bgs} \cite{Lin04} \cite{vf03})
strongly use the fact that vorticity is non-streching in $2$D and therefore do
not work in $3$D due to the vorticity stretching effect.

Below, we sketch the main ideas in the proof of Theorem \ref{T:UnstableM}. The
main difficulty in constructing the unstable manifolds for the Euler equation
lies in the fact that a derivative loss occurs in the nonlinear terms while
the linearized flow does not have the smoothing property. We will prove
Theorem \ref{T:UnstableM} mainly by considering the Euler equation (E) in the
Lagrangian coordinates. In a seminal paper \cite{ar66}, V. Arnold pointed out
that the incompressible Euler equation can be viewed as the geodesic equation
on the group of volume preserving diffeomorphisms. This point of view has been
adopted and developed by several authors in their work on the Euler equations,
such as \cite{em70, sh85, br99, sz08a, sz08b} to mention a few.

On the one hand, the main advantage of this approach is that (E) on a fixed
domain as in this paper becomes a smooth infinite dimensional ODE \cite{em70}
on the tangent bundle of the Lie group $\mathcal{G}$ of volume preserving
diffeomorphisms of $\Omega$ and thus the difficulty arising from the loss of
regularity disappears. A side remark is that this is in contrast with the
Euler equation with free boundaries \cite{sz08a, sz08b} where the Riemannian
curvature of the infinite dimensional manifolds of volume preserving manifolds
are unbounded operators and the Euler equation can not be considered as
infinite dimensional ODEs. As a clarification, by saying that the Euler
equation on fixed domains defines an ODE, we mean that it corresponds to a
vector field which is everywhere defined and smooth on an infinite dimensional
manifold. In this sense, evolutionary PDEs involving unbounded operators such
as heat equation or wave equation do not define infinite dimensional ODEs.

On the other hand, in the Lagrangian coordinates, the steady state $v_{0}$
generates a special geodesic $u_{0}(t)$ which coincides with the integral
curve starting at the identity map of the right invariant vector field on
$\mathcal{G}$ generated by $v_{0}$. By carefully using local coordinates along
this orbit generated by the group symmetry, we further transform the localized
Euler equation into a weakly nonlinear non-autonomous ODE with linear
exponential dichotomy. One notes that the multiplication on this Lie group
$\mathcal{G}$ -- the composition between $W^{k, q}$ volume preserving maps --
is continuous, but not smooth, see Proposition \ref{P:translation}. The
assumption $v_{0} \in W^{k+r, q}$ ensures that the localization we choose
possesses certain smoothness. Moreover, compared to the usual exponential
dichotomy enjoyed by many ODEs and even PDEs (see for example \cite{CL88}),
the exponential dichotomy here has the defect that the angle between the
associated invariant subspaces may not have a uniform positive lower bound in
time due to the possible growth of the linearized ODE flow of the vector field
$v_{0}$. Moreover, this type of unwanted growth in $t$ may also appear in the
norms of the nonlinearities. Assumption (A3) is used to overcome this
non-uniformity in $t$ of the dichotomy in the construction of the unstable
manifolds via the method based on the Lyapunov-Perron integral equations. Here
we have to take the advantage of the fact that solutions on the unstable
manifolds decay to $v_{0}$ exponentially as $t\rightarrow-\infty$. Since
solutions on the center manifolds do not have similar decay properties, we
still can not construct center manifolds of stead states.

\begin{remark}
The assumption (A2) is the standard linear exponential dichotomy for
constructing invariant manifolds. The extra gap assumption (A3) might be
technical, but appears rather natural in our approach. In the proof, we change
back and forth between Lagrangian and Eulerian coordinates, and each
transformation induces a factor $e^{\mu_{0}t}$ in the estimates. The extra gap
$\lambda_{u}-\lambda_{cs}>K_{0}\mu_{0}$ guarantees that after all these
transformations, the exponential dichotomy of unstable and central-stable
parts still persists.
\end{remark}

The method introduced in this paper provides a general approach to construct
unstable manifolds for many other continuum models in fluid and plasmas. In
these models, the loss of derivative is also due to nonlinear terms from the
material derivative. By working on Lagrangian coordinates, we can again
overcome such a loss of derivative and the existence of unstable manifolds is
conceivable with sufficient spectral gap. We are using this approach to
construct unstable manifolds for density-dependent Euler equations and the
Vlasov-Poisson system for collisionless plasmas.

\section{Proof of Theorem \ref{T:UnstableM}}

\label{S:MProof}

Lagrangian coordinates is a standard tool in studying the Euler equation. In
Subsection \ref{SS:Lag} and \ref{SS:ODE}, we will present the manifold
structure of the set $\mathcal{G}$ of the Lagrangian maps and the ODE nature
of (E) on $T\mathcal{G}$. These general results have actually been proved even
for Euler equations defined on Riemannian manifolds in \cite{em70} in a rather
geometric language. However, we need to establish a more concrete framework
along with more detailed estimates to be used in the construction of local
invariant manifolds in Subsections \ref{SS:non-auto} -- \ref{SS:InMa}, which
is done in a more directly equation based manner in Subsections \ref{SS:Lag}
and \ref{SS:ODE}. In Subsection \ref{SS:non-auto}, we rewrite in the
Lagrangian coordinates the localized Euler equation in a neighborhood of the
solution curve generated by $v_{0}$ and in Subsection \ref{SS:exp-di}, the
linear exponential dichotomy is given. The unstable integral manifold
corresponding the linear exponential dichotomy is constructed in Subsection
\ref{SS:t-InMa} and then finally the unstable manifold in the Eulerian
coordinates is obtained in Subsection \ref{SS:InMa}.

Throughout this section, we will use $K>0$ as a generic constant depending
only on $r$ and $k$ and $C>0$ only on $n, r, k, q, v_{0}$. Both $K$ and $C$
may change from line to line. We will use $D$ or $\nabla$ to denote the
differentiation with respect to physical variables in $\Omega$ and
$\mathcal{D}$ the Fr\'echet differentiations in function spaces.

\subsection{Lagrangian coordinates and the Lie group of volume preserving
maps}

\label{SS:Lag}

Let $u(t,\cdot):\Omega\rightarrow\Omega$ be the Lagrangian coordinate map
defined by
\begin{equation}
u(0,y)=y\quad\text{ and }u_{t}(t,y)=v(t,u(t,y)). \label{E:LagM}%
\end{equation}
In particular, let $u_{0}(t,y)$ be the Lagrangian map of the steady vector
field $v_{0}(x)$. Throughout the paper, we fix a constant $\mu>\mu_{0}\geq0$
such that
\begin{equation}
\lambda_{u}-\lambda_{cs}>K_{0}\mu\label{E:mu0}%
\end{equation}
where $\mu_{0}$ is the Lyapunov exponent of $v_{0}$. From the definition of
the Lyapunov exponent, we have
\begin{equation}
|u_{0}(t,\cdot)|_{C^{l}}+|\big(u_{0}(t,\cdot)\big)^{-1}|_{C^{l}}\leq
Ce^{l\mu|t|},\quad t\in\mathbb{R},\;0\leq l\leq k+r \label{E:LagM0}%
\end{equation}
for some $C>0$ independent of $t$. This possible exponential growth of the
norm of $u_{0}$ makes the problem much more subtle than the unusual
constructions of the local invariant manifolds in differential equations.

Since the flow is incompressible and $k>1+\frac{n}{q}$, we have for any
$t\in\mathbb{R}$,
\begin{equation}
u(t,\cdot)\in\mathcal{G}\triangleq\{\phi\in W^{k,q}(\Omega,\mathbb{R}^{n}%
)\mid\phi\text{ is a diffeomorphism, }\det(D\phi)\equiv1,\phi(\partial
\Omega)=\partial\Omega\}. \label{E:LieG}%
\end{equation}
Clearly the composition makes $\mathcal{G}$ a group. We will show that
$\mathcal{G}$ is an infinite dimensional submanifold of $W^{k,q}%
(\Omega,\mathbb{R}^{n})$. This will be our configuration space when the Euler
equation is written in the Lagrangian coordinates. We will work with local
coordinates on $\mathcal{G}$.

Formally, the tangent space of $\mathcal{G}$ is given by
\begin{equation}
\label{E:tangentID}T_{id} \mathcal{G} = W_{Euler}^{k, q}, \qquad T_{\phi
}\mathcal{G} = \{w \mid w \circ\phi^{-1} \in W_{Euler}^{k, q} \} \quad
\forall\phi\in\mathcal{G},
\end{equation}
where $W_{Euler}^{k, q}$ defined in \eqref{E:phaseS} is the phase space of the
velocity fields of (E). From the Hodge decomposition, a complementary space of
$W_{Euler}^{k, q}$ in $W^{k, q} (\Omega, \mathbb{R}^{n})$ is given by
\begin{equation}
\label{E:ortho}(W_{Euler}^{k, q})^{\perp}= \{ \nabla h \mid h \in W^{k+1, q}
(\Omega, \mathbb{R})\}.
\end{equation}
Here the orthogonality is in the sense that
\[
\int_{\Omega}w\cdot\nabla h dx = 0, \qquad\forall w \in W_{Euler}^{k, q}, \; h
\in W^{k+1, q} (\Omega, \mathbb{R}).
\]
It is clear that both $W_{Euler}^{k, q}$ and $(W_{Euler}^{k, q})^{\perp}$ are
closed subspaces of $W^{k, q} (\Omega, \mathbb{R}^{n})$ and
\[
W^{k, q} (\Omega, \mathbb{R}^{n}) = W_{Euler}^{k, q} \oplus(W_{Euler}^{k,
q})^{\perp}.
\]
In fact, given any $X \in W^{k, q} (\Omega, \mathbb{R}^{n})$, let
\begin{equation}
\label{E:Helmholtz}w = X - \nabla h
\end{equation}
where $h$ is the solution of
\begin{equation}
\label{E:Hodge1}\Delta h = \nabla\cdot X \quad\text{ in } \Omega\qquad
\nabla_{N} h = X \cdot N \quad\text{ on } \partial\Omega,
\end{equation}
then obviously $X= w+\nabla h$ with $w \in W_{Euler}^{k, q}$ and this verifies
the direct sum.
Locally near the identity map $id$, we will write $\mathcal{G}$ as the graph
of a smooth mapping from $W_{Euler}^{k, q}$ to $(W_{Euler}^{k, q})^{\perp}$
and thus $\mathcal{G}$ is rigorously a smooth manifold. Let $B_{\delta}(\cdot)
$ denote the ball of radius $\delta$ centered at $0$ in the corresponding
Banach space.

\begin{proposition}
\label{P:coordCG} There exists $\delta_{0}>0$ and a smooth mapping
$\Psi:B_{\delta_{0}}(W_{Euler}^{k,q})\rightarrow W^{k,q}(\Omega,\mathbb{R}%
^{n})$, such that $\Psi(0)=id$, $\mathcal{D}\Psi(0)=I$, $\Psi(w)-w\in
(W_{Euler}^{k,q})^{\perp}$ for any $w\in B_{\delta_{0}}(W_{Euler}^{k,q})$,
and
\[
\Big(id+B_{\frac{\delta_{0}}{2}}\big(W^{k,q}(\Omega,\mathbb{R}^{n}%
)\big)\Big)\cap\mathcal{G}\,\subset\,\{\Psi(w)\mid w\in B_{\delta_{0}%
}(W_{Euler}^{k,q})\}\,\subset\,\Big(id+B_{2\delta_{0}}\big(W^{k,q}%
(\Omega,\mathbb{R}^{n})\big)\Big)\cap\mathcal{G}.
\]

\end{proposition}

\begin{proof}
Since $\partial\Omega$ is a smooth compact hypersurface in $\mathbb{R}^{n}$,
the distance function to $\partial\Omega$ is smooth in a neighborhood of
$\partial\Omega$. Let $d:\mathbb{R}^{n}\rightarrow\mathbb{R}$ be a smooth
function with compact support such that it coincides with this distance
function in a neighborhood of $\partial\Omega$. Consider the mapping
\[
G:\;W^{k,q}(\Omega,\mathbb{R}^{n})\rightarrow Y\triangleq\{(f,g)\in
W^{k-1,q}(\Omega,\mathbb{R})\times W^{k-\frac{1}{q},q}(\partial\Omega
,\mathbb{R})\mid\int_{\partial\Omega}g\,dS=0\}
\]
defined as
\[
G(\phi)=\big(\det(D\phi),\;(d\circ\phi\big)|_{\partial\Omega}-\frac
{1}{|\partial\Omega|}\int_{\partial\Omega}d\circ\phi\,dS)
\]
where $|\partial\Omega|$ denotes the area of $\partial\Omega$. Obviously, $G$
is a smooth mapping and $G|_{\mathcal{G}}\equiv(1,0)$. Moreover, suppose
$\phi\in U$ and $G(\phi)=(1,0)$ where $U$ is a neighborhood of the identity
map $id$ in $W^{k,q}(\Omega,\mathbb{R}^{n})$. Then $\phi$ is a diffeomorphism
from $\Omega$ to its image and
\[
|\phi(\Omega)|=|\Omega|\quad\text{ and }\quad d\circ\phi=a\triangleq\frac
{1}{|\partial\Omega|}\int_{\partial\Omega}d\circ\phi\,dS
\]
imply that $d\circ\phi\equiv a=0$. Otherwise $a\neq0$ would imply that
$\phi(\Omega)$ either strictly covers $\Omega$ or is strictly contained in
$\Omega$, either of which contradicts with the first identity above.
Therefore
\[
U\cap\mathcal{G}=U\cap G^{-1}\{(1,0)\}.
\]
It is easy to compute
\[
\mathcal{D}G(id)X=(\nabla\cdot X,\;X\cdot N-\frac{1}{|\partial\Omega|}%
\int_{\partial\Omega}X\cdot N\,dS)\in Y,\quad\forall\,X\in W^{k,q}%
(\Omega,\mathbb{R}^{n}).
\]
From the standard theory of elliptic problems with Neumann boundary
conditions,
\[
DG(id)\nabla h=(\Delta h,\,\nabla_{N}h-\frac{1}{|\partial\Omega|}\int_{\Omega
}\Delta h\,dx)
\]
is an isomorphism from $(W_{Euler}^{k,q})^{\perp}$ to $Y$. Therefore the
proposition follows from the Implicit Function Theorem.
\end{proof}

Near any $\phi_{0} \in\mathcal{G}$, the local coordinate map $\Psi$ composed
with the right translation, $\Psi(\cdot) \circ\phi_{0}$, gives a smooth local
coordinate map near $\phi_{0}$. Therefore $\mathcal{G}$ is a Banach
submanifold with the model space $W_{Euler}^{k, q}$. It is well-known
\cite{em70} that $\mathcal{G}$ is not such a standard Lie group as one needs
to be careful with the smoothness of the left translations. Let $\mathcal{C}%
(\phi_{1}, \phi_{2}) = \phi_{1} \circ\phi_{2}$ and it is straightforward to verify

\begin{proposition}
\label{P:translation} $\mathcal{C} \in C^{l} \Big( \big(\mathcal{G} \cap
W^{k+l, q}(\Omega, \mathbb{R}^{n})\big) \times\mathcal{G}, \mathcal{G}\Big)$
and, for any $\phi_{2} \in\mathcal{G}$, the right translation $\mathcal{C}%
(\cdot, \phi_{2}) \in C^{\infty}(\mathcal{G})$.
\end{proposition}

\subsection{Euler equation as an ODE on $T\mathcal{G}$}

\label{SS:ODE}

It is well-known that the pressure $p$ can be represented in terms of the
velocity field $v$. In fact, by taking the inner product of $v_{t}$ and $N$ on
$\partial\Omega$ and the divergence of the $v_{t}$ and using $\nabla\cdot v=0$
in $\Omega$ and $N\cdot V=0$ on $\partial\Omega$, we obtain
\begin{equation}%
\begin{cases}
-\Delta p=\Sigma_{i,j=1}^{n}\partial_{i}(v^{j}\partial_{j}v^{i})=\Sigma
_{i,j=1}^{n}\partial_{i}v^{j}\partial_{j}v^{i}=\text{tr}(Dv)^{2}\qquad &
\text{ in }\Omega\\
\nabla_{N}p=-N\cdot(v\cdot\nabla)v=-\nabla_{v}(v\cdot N)+\nabla_{v}N\cdot
v=v\cdot\Pi(v) & \text{ on }\partial\Omega
\end{cases}
\label{E:pressure1}%
\end{equation}
where the symmetric operator $\Pi\in C^{\infty}\big(\Omega,L(T\partial
\Omega)\big)$ is the second fundamental form of $\partial\Omega$ defines as
$\Pi(x)(\tau)=\nabla_{\tau^{\top}}N$ with $\tau^{\top}\in T_{x}\partial\Omega$
being the tangential component of $\tau$.

Based on the form of the pressure, we define the symmetric bounded bilinear
mapping $B:(W^{k,q}(\Omega,\mathbb{R}^{n}))^{2}\rightarrow W^{k,q}%
(\Omega,\mathbb{R}^{n})$ as
\[
\mathcal{B}(X_{1},X_{2})=\nabla\gamma
\]
where
\[%
\begin{cases}
-\Delta\gamma=\text{tr}(DX_{1}DX_{2})=\Sigma_{i,j=1}^{n}\partial_{i}X_{1}%
^{j}\partial_{j}X_{2}^{i}\qquad & \text{ in }\Omega\\
\nabla_{N}\gamma=X_{1}\cdot\Pi(X_{2}) & \text{ on }\partial\Omega.
\end{cases}
\]
The boundedness of $\mathcal{B}$ is clear from the standard elliptic theory.
Note that here we do not assume $\nabla\cdot X_{1,2}=0$ in $\Omega$ or $N\cdot
X_{1,2}$ on $\partial\Omega$ in the definition of $\mathcal{B}$. According to
the Hodge decomposition given through \eqref{E:Hodge1} and a similar
calculation as in \eqref{E:pressure1}, it holds that
\begin{equation}
DX_{1}(X_{2})+\mathcal{B}(X_{1},X_{2})=(X_{2}\cdot\nabla)X_{1}+\mathcal{B}%
(X_{1},X_{2})\in W_{Euler}^{k,q},\qquad\forall X_{1},\,X_{2}\in W_{Euler}%
^{k,q}. \label{E:Hodge2}%
\end{equation}
As a side remark, it indicates that, when embedded in $L^{2}(\Omega
,\mathbb{R}^{n})$, $\mathcal{B}$ is the second fundamental form of
$\mathcal{G}$ at $id$ which can be rigorously verified through a standard procedure.

Define the mapping $\mathcal{P}: \mathcal{G} \times\left(  W^{k, q} (\Omega,
\mathbb{R}^{n})\right)  ^{2} \to W^{k, q} (\Omega, \mathbb{R}^{n})$ as
\begin{equation}
\label{E:CP}\mathcal{P}(\phi, X_{1}, X_{2})= \mathcal{B}(X_{1} \circ\phi^{-1},
X_{2} \circ\phi^{-1}) \circ\phi.
\end{equation}
We also define the projection $Q: \mathcal{G} \times W^{k, q} (\Omega,
\mathbb{R}^{n}) \to T\mathcal{G}$ as
\begin{equation}
\label{E:Helmholtz1}Q(\phi, X) = (X \circ\phi^{-1} - \nabla h) \circ\phi\in
T_{\phi}\mathcal{G}%
\end{equation}
where $\nabla h$ for $X \circ\phi^{-1}$ is defined in \eqref{E:Helmholtz}.
Obviously, $\mathcal{P}$ is symmetrically bilinear in $X_{1}$ and $X_{2}$. As
in \eqref{E:Hodge2}, it holds
\begin{equation}
\label{E:Hodge3}\Big(\big(D(X_{1} \circ\phi^{-1}) \big) \circ\phi\Big)(X_{2})+
\mathcal{P}(\phi, X_{1}, X_{2}) \in T_{\phi}\mathcal{G} = \{w: \Omega
\to\mathbb{R}^{n} \mid w \circ\phi^{-1} \in W_{Euler}^{k, q} \}.
\end{equation}
In fact, \eqref{E:Hodge2} also leads to that, when embedded in $L^{2}$,
$\mathcal{P}(\phi, \cdot, \cdot)$ is the second fundamental form of
$\mathcal{G}$ at $\phi$. Euler equation (E) and \eqref{E:pressure1} imply that
the Euler equation (E) takes the form in the Lagrangian coordinates
\begin{equation}
\label{E:EulerL}u_{tt} + \mathcal{P}(u, u_{t}, u_{t}) =0, \quad u(t)
\in\mathcal{G}.
\end{equation}
Moreover, for any $\phi_{0} \in\mathcal{G}$, we have
\begin{equation}
\label{E:RT}\mathcal{P}(\phi\circ\phi_{0}, X_{1} \circ\phi_{0}, X_{2}
\circ\phi_{0}) = \mathcal{P}(\phi, X_{1}, X_{2}) \circ\phi_{0},
\end{equation}
i.e. $\mathcal{P}$ is invariant under the right translation. Therefore,
\begin{equation}
\label{E:symm}u(t) \circ\phi_{0} \text{ is also a solution for any solution }
u(t) \text{ of \eqref{E:EulerL}}.
\end{equation}
The following proposition states that $\mathcal{P}$ is a smooth mapping and
thus has no regularity loss which is far from trivial even though
$\mathcal{B}$ is a bounded bilinear operator. To see this, one note that even
the dependence of the term $X_{1} \circ\phi\in W^{k, q}$ on $\phi\in W^{k, q}$
is not smooth unless $X_{1}$ belongs to a space of better regularity. The
proof of the proposition is essentially a careful analysis of the commutator
between the operations of $\mathcal{B}$ and the composition by $\phi
\in\mathcal{G}$.

\begin{proposition}
\label{P:CP} $\mathcal{P}: \mathcal{G} \times\left(  W^{k, q} (\Omega,
\mathbb{R}^{n})\right)  ^{2} \to W^{k, q} (\Omega, \mathbb{R}^{n})$ and $Q:
\mathcal{G} \to L\big(W^{k, q} (\Omega, \mathbb{R}^{n})\big)$ are $C^{\infty}$
and, for any $m>0$, there exist $C_{0}, K>0$ depending only on $m$ and $k$
such that
\begin{align}
&  |(\mathcal{D}_{\phi})^{m} \mathcal{P}(\phi, X_{1}, X_{2})|_{W^{k, q}} \le
C_{0} |D \phi|_{W^{k-1, q}}^{K} |X_{1}|_{W^{k, q}} |X_{2}|_{W^{k, q}%
}\label{E:CP1}\\
&  |(\mathcal{D}_{\phi})^{m} Q(\phi, X)|_{W^{k, q}} \le C_{0} |D
\phi|_{W^{k-1, q}}^{K} |X|_{W^{k, q}}.
\end{align}

\end{proposition}

Here $(\mathcal{D}_{\phi})^{m} \mathcal{P}$ should be considered as a
multilinear operator from $T_{\phi}\mathcal{G}$ to $W^{k, q} (\Omega)$. Since
$\mathcal{P}$ is bilinear in $X_{1,2}$, the bounds on the derivatives of
$\mathcal{P}$ with respect to $X_{1,2}$ follow from \eqref{E:CP1} for $m=0$.
Also, we note that $|D \phi|_{W^{k-1, q}}\ge1$ since $k-1 > \frac nq$ and $det
D \phi\equiv1$.

\begin{proof}
We will present only the proof for $\mathcal{P}$ as the one for $Q$ follows
through almost exactly the same (or even slightly simpler) procedure. Due to
the invariance \eqref{E:RT} of $\mathcal{P}$ under the right translation, we
only need to show its smoothness near $\phi=id$. As $\phi$ belongs to the
manifold $\mathcal{G}$, the smoothness of $\mathcal{P}$ is equivalent to the
smoothness of $\mathcal{P}\big(\Phi(w),X_{1},X_{2}\big)$ with respect to $w\in
W_{Euler}^{k,q}$ and $X_{1,2}\in W^{k,q}(\Omega,\mathbb{R}^{n})$ for any
smooth local coordinate map $\Phi:B_{\delta}(W_{Euler}^{k,q})\rightarrow
\mathcal{G}$. Moreover, to prove the Fr\'{e}chet smoothness of $\mathcal{P}$,
it suffices to show the G\^{a}teaux differentiability of $\mathcal{P}$ up to
the $m$-th order for any $m>0$, which would imply that G\^{a}teaux derivative
$\mathcal{D}^{m-1}\mathcal{P}$ is continuous and thus it is also the
$(m-1)$-th Fr\'{e}chet derivative. To show the G\^{a}teaux differentiability
up to the $m$-th order, it suffices to prove the smoothness of
\[
\mathcal{P}\big(\phi(s_{1},\ldots,s_{m}),X_{1}(s_{1},\ldots,s_{m}),X_{2}%
(s_{1},\ldots,s_{m})\big)\in W^{k,q}(\Omega,\mathbb{R}^{n})
\]
for any $\phi(s_{1},\ldots,s_{m})\in\mathcal{G}\subset W^{k,q}(\Omega
,\mathbb{R}^{n})$ and $X_{1,2}(s_{1},\ldots,s_{m})\in W^{k,q}(\Omega
,\mathbb{R}^{n})$ with smooth parameters $(s_{1},\ldots,s_{m})\in
U\subset\mathbb{R}^{m}$. We will show by induction
\begin{equation}
Z\triangleq\partial_{s_{1}}\ldots\partial_{s_{m}}\mathcal{P}(\phi,X_{1}%
,X_{2})|_{s_{1}=\ldots=s_{m}=0}\in W^{k,q}(\Omega,\mathbb{R}^{n}) \tag{ML}%
\end{equation}
and obtain its bound in the form of \eqref{E:CP1}.

The boundedness of the bilinear transformation $\mathcal{B}$ implies
\eqref{E:CP1} and (ML) for $m=0$. Assume (ML) and \eqref{E:CP1} hold for $0
\le m<m_{0}$ and we will prove it for $m=m_{0}$. Let $\mathcal{P}(\phi, X_{1},
X_{2}) = (\nabla\gamma) \circ\phi$, where $\gamma$, defined as in the
definition of $\mathcal{B}$, depends on $s_{1}, \ldots, s_{m_{0}}$ through
$\phi$, $X_{1}$, and $X_{2}$. In the following, since it will be much easier
to carry out some calculations in the Eulerian coordinates, let
\[
\tilde X_{j} = X_{j} \circ\phi^{-1}, \quad j=1,2; \qquad\tau_{i} =
(\partial_{s_{i}} \phi) \circ\phi^{-1}, \quad{\boldsymbol{D}}_{s_{i}} =
\partial_{s_{i}} + \nabla_{\tau_{i}}, \qquad i=1, \ldots, m_{0}%
\]
and
\[
\tilde Z = {\boldsymbol{D}}_{s_{1}} \ldots{\boldsymbol{D}}_{s_{m_{0}}}
\nabla\gamma= Z \circ\phi^{-1}.
\]
Clearly, it is sufficient to show $\tilde Z \in W^{k, q}$, which will be
achieved by studying its normal component on $\partial\Omega$, divergence, and curl.

To start, for any vector field $W\in W^{k-1+m_{0},q}(\Omega,\mathbb{R}^{n})$,
let $Y(y)=\big(D\phi(y)\big)^{-1}W\big(\phi(y)\big)$, or equivalently
$D\phi(Y)=W\circ\phi$, then
\[
D\phi(\partial_{s_{i}}Y)=-D\phi_{s_{i}}(Y)+\big((DW)\circ\phi\big)\phi_{s_{i}%
}\implies\partial_{s_{i}}Y\in W^{k-1,q}(\Omega,\mathbb{R}^{n}).
\]
Differentiating the above identity one more time implies
\[
D\phi(\partial_{s_{j}s_{i}}Y)=-D\phi_{s_{j}}(\partial_{s_{i}}Y)-D\phi_{s_{i}%
}(\partial_{s_{j}}Y)-D\phi_{s_{j}s_{i}}(Y)+\big((DW)\circ\phi\big)\phi
_{s_{j}s_{i}}+\big((D^{2}W)\circ\phi\big)(\phi_{s_{j}},\phi_{s_{i}})
\]
and thus the smoothness of $\phi$ and $W$ yield
\[
\partial_{s_{j}s_{i}}Y\in W^{k-1,q}(\Omega,\mathbb{R}^{n})\quad\text{ if
}m_{0}\geq2.
\]
Repeating this procedure we obtain $\partial_{s_{i_{1}}\ldots s_{i_{m_{0}}}%
}Y\in W^{k-1,q}(\Omega,\mathbb{R}^{n})$ inductively. Changing from the
Eulerian coordinates to the Lagrangian coordinates, we have
\[
\big(\nabla_{W}\tilde{Z}-{\boldsymbol{D}}_{s_{1}}\ldots{\boldsymbol{D}%
}_{s_{m_{0}}}(D^{2}\gamma(W))\big)\circ\phi=\nabla_{Y}Z-\partial_{s_{1}}%
\ldots\partial_{s_{m_{0}}}\nabla_{Y}\big((\nabla\gamma)\circ\phi\big).
\]
Applying the induction assumption to the last term above and using the
commutator formula $[\partial_{s_{i}},\nabla_{Y}]=\nabla_{\partial_{s_{i}}Y}$
to move the $\nabla_{Y}$ to the outside to produce $\nabla_{Y}Z$, we obtain
\begin{equation}
\big(\nabla_{W}\tilde{Z}-{\boldsymbol{D}}_{s_{1}}\ldots{\boldsymbol{D}%
}_{s_{m_{0}}}(D^{2}\gamma(W))\big)\circ\phi=\nabla_{Y}Z-\partial_{s_{1}}%
\ldots\partial_{s_{m_{0}}}\nabla_{Y}\big((\nabla\gamma)\circ\phi\big)\in
W^{k-1,q}. \label{E:Z1}%
\end{equation}

Taking $W=e_{1},\ldots,e_{n}$ of the standard basis of $\mathbb{R}^{n}$,
\eqref{E:Z1} and the definition of $\mathcal{P}$ imply the curl $\nabla
\times\tilde{Z}$ contains only the commutators terms and thus satisfies
\[
\nabla\times\tilde{Z}\in W^{k-1,q}.
\]
Similarly the divergence satisfies
\[
\nabla\cdot\tilde{Z}+{\boldsymbol{D}}_{s_{1}}\ldots{\boldsymbol{D}}_{s_{m_{0}%
}}(\text{tr}(D\tilde{X}_{1})(D\tilde{X}_{2}))\in W^{k-1,q}.
\]
Expanding the above last term using the product rule on ${\boldsymbol{D}%
}_{s_{i}}$, it consists of terms in the form of
\[
({\boldsymbol{D}}_{s_{i_{1}}}\ldots{\boldsymbol{D}}_{s_{i_{m}}}\partial
_{l_{1}}\tilde{X}_{1}^{l_{2}})({\boldsymbol{D}}_{s_{j_{1}}}\ldots
{\boldsymbol{D}}_{s_{j_{m_{0}-m}}}\partial_{l_{2}}\tilde{X}_{2}^{l_{1}}%
),\quad\{i_{1},\ldots,i_{m},\,j_{1},\ldots j_{m_{0}-m}\}=\{1,\ldots,m_{0}\}.
\]
Move $\partial_{l_{1}}$ and $\partial_{l_{2}}$ to the outside in the same
fashion as in the derivation of \eqref{E:Z1} (replacing $W$ by $e_{l_{1,2}}$
and $\nabla\gamma$ by $\tilde{X}_{1,2}^{l_{1,2}}$) and using the smoothness of
$X_{1,2}$ in $s_{1},\ldots,s_{m_{0}}$, it is easy to obtain ${\boldsymbol{D}%
}_{s_{1}}\ldots{\boldsymbol{D}}_{s_{m_{0}}}(\text{tr}(D\tilde{X}_{1}%
)(D\tilde{X}_{2}))\in W^{k-1,q}$ and thus
\[
\nabla\cdot\tilde{Z}\in W^{k-1,q}.
\]
Finally, using
\[
{\boldsymbol{D}}_{s_{i}}N=\Pi(\tau_{i}),\quad{\boldsymbol{D}}_{s_{j}%
}{\boldsymbol{D}}_{s_{i}}N=(\nabla_{\tau_{j}}\Pi)(\tau_{i})+\Pi(\phi
_{s_{j}s_{i}}\circ\phi^{-1}),\quad\ldots
\]
the induction assumption, and the assumption $N\cdot\nabla\gamma=0$ on
$\partial\Omega$ in the definition of $\mathcal{P}$ it is straight forward to
obtain
\[
\tilde{Z}\cdot N=N\cdot{\boldsymbol{D}}_{s_{1}}\ldots{\boldsymbol{D}%
}_{s_{m_{0}}}\nabla\gamma\in W^{k-\frac{1}{q},q}(\partial\Omega,\mathbb{R})
\]
in the same fashion. Therefore, the standard estimates in elliptic theory
implies that (ML) holds for $m=m_{0}$. Moreover, inequality \eqref{E:CP1}
follows from the observation that the composition by $\phi$ or $\phi^{-1}$
only produces terms like $|\phi|_{W_{k,q}}^{l}$ in the estimates of the
$W^{k,q}$ norms and thus the proof of the proposition is complete.
\end{proof}

Proposition \ref{P:CP} provides the key element for us to prove that
\eqref{E:EulerL} is a smooth second order ODE on the infinite dimensional
configuration manifold $\mathcal{G}$.

\begin{proposition}
\label{P:ODE} For any $u_{0} \in\mathcal{G}$ and $w_{0} \in T_{u_{0}}
\mathcal{G}$, the initial value problem of the Euler equation \eqref{E:EulerL}
has a unique solution $\big(u(t), u_{t}(t)\big) \in T\mathcal{G}$, locally in
time, depending on $(t, u_{0}, w_{0})$ smoothly.
\end{proposition}

\begin{proof}
Due to the right translation of \eqref{E:EulerL} given in \eqref{E:symm}, we
may assume $u_{0}$ belongs to a small neighborhood of $id$. From Proposition
\ref{P:coordCG},
\[
(w, \nabla h) \to\Psi(w) + \nabla h
\]
is a local diffeomorphism from $W_{Euler}^{k, q} \times(W_{Euler}^{k,
q})^{\perp}$ to $W^{k, q} (\Omega, \mathbb{R}^{n})$. Taking the $\Psi(w)$
component, we obtain a smooth $\Phi: id + B_{\delta}\big(W^{k, q} (\Omega,
\mathbb{R}^{n}) \big) \to\mathcal{G}$ such that
\begin{equation}
\label{E:Phi}u- \Phi(u) \in(W_{Euler}^{k, q})^{\perp}, \; \forall u \in id +
B_{\delta}\big(W^{k, q} (\Omega, \mathbb{R}^{n}) \big)
\end{equation}
and
\[
\Phi(u) =u, \; \forall\, u \in\mathcal{G} \cap\Big(id +B_{\delta}\big(W^{k, q}
(\Omega, \mathbb{R}^{n}) \big)\Big).
\]
Consider a modification of \eqref{E:EulerL}
\begin{equation}
\label{E:MEulerL}u_{tt} + \mathcal{P}\big( \Phi(u), u_{t}, \mathcal{D} \Phi(u)
u_{t}\big) =0, \quad u\in B_{\delta}\big(W^{k, q} (\Omega, \mathbb{R}%
^{n})\big)\; u_{t} \in B_{\delta}\big(W^{k, q} (\Omega, \mathbb{R}^{n})\big).
\end{equation}
From the smoothness of $\Phi$ and $\mathcal{P}$, equation \eqref{E:MEulerL} is
a smooth ODE defined on an open subset of an infinite dimensional Banach space
$\big(W^{k, q} (\Omega, \mathbb{R}^{n})\big)^{2}$ and thus is locally
well-posed with smooth dependence on the initial value. Moreover, any solution
$u(t)$ of \eqref{E:MEulerL} satisfying $u(t) \in\mathcal{G}$ for all $t$ also
solves \eqref{E:EulerL}. Therefore, to complete the proof, we only need to
show that, if the initial data is given on $T\mathcal{G}$, then the solution
of \eqref{E:MEulerL} also stays on $T\mathcal{G}$, i. e. $u(t) \in\mathcal{G}$
and $u_{t} (t) \in T_{u(t)} \mathcal{G}$. Let $v = u_{t} \circ\big(\Phi
(u)\big)^{-1}$. Equation \eqref{E:MEulerL} yields
\[
v_{t} \circ\Phi(u) + \big((D v) \circ\Phi(u)\big) \big( \mathcal{D} \Phi(u)
u_{t}\big) = u_{tt} = - \mathcal{P}\big( \Phi(u), u_{t}, \mathcal{D} \Phi(u)
u_{t}\big)
\]
and thus \eqref{E:Hodge3} implies $v_{t} \circ\Phi(u) \in T_{\Phi(u)}
\mathcal{G}$ or equivalently $v_{t} \in T_{id} \mathcal{G} = W_{Euler}^{k, q}%
$. Therefore $v(t) \in W_{Euler}^{k, q}$, and thus $u_{t} = v \circ\Phi(u) \in
T_{\Phi(u)} \mathcal{G}$, for all $t$ follows from the initial assumption.
Consequently
\[
\big( u - \Phi(u) \big)_{t} = \big(I-\mathcal{D}\Phi(u)\big) u_{t} = 0
\]
where we also used that \eqref{E:Phi} implies $\mathcal{D} \Phi(u) X = X$ for
any $X \in T_{\Phi(u)} \mathcal{G}$. Therefore $u= \Phi(u) \in\mathcal{G}$ and
the proof completes.
\end{proof}

According to \eqref{E:symm}, the second order ODE \eqref{E:EulerL} defined on
the Lie group $\mathcal{G}$ is invariant under the right translation. The
standard procedure of taking $v=u_{t}\circ u^{-1}$ reduces it to a first order
equation on the corresponding Lie algebra $W_{Euler}^{k,q}=T_{id}\mathcal{G}$,
which turns out to be the usual form (E) of the Euler equation in the Eulerian
coordinates. However, according to Proposition \ref{P:translation}, the
composition as the multiplication operation on the group $\mathcal{G}$ is not
smooth, this procedure induces the loss of one order of spatial derivative and
make (E) into a PDE, i. e. the right side of (E) does not define a smooth
vector field on $W_{Euler}^{k,q}$.

\subsection{Euler equation near $v_{0}$ as a non-autonomous ODE on
$T\mathcal{G}$}

\label{SS:non-auto}

Even though Euler equation is equivalent to an infinite dimensional ODE on
$T\mathcal{G}$, the Lagrangian frame work also brings two complications:

\begin{itemize}
\item $\mathcal{G}$ is not flat, which means that we may have to carry out the
analysis in local coordinates on $\mathcal{G}$ and

\item the steady velocity field $v_{0}(x)$ of the Euler equation (E)
corresponds to a dynamic solution $\Big(u_{0}(t, y), u_{0t} (t, y) =
v_{0}\big( u(t, y)\big)\Big)$ of \eqref{E:EulerL}.
\end{itemize}

It is natural to look for ways to take the advantage of the group structure of
$\mathcal{G}$ to reduce \eqref{E:EulerL}, localized near $u_{0}(t)$, to a
non-autonomous second order ODE defined in local coordinate neighborhood of
$id$. Based on the comments at the end of Subsection \ref{SS:ODE}, taking $w=u
\circ u_{0}^{-1}$ for $u(t)$ near $u_{0}(t)$ would result in loss of
regularity, which can also be seen explicitly through the simple calculation
$u_{t} = w_{t} \circ u_{0} + (Dw \circ u_{0}) (v_{0} \circ u_{0})$ leading to
$w_{t} \notin W^{k, q}$ due to the presence of $Dw$. Instead, for any solution
$\big(u(t), u_{t}(t) = v(t) \circ u(t)\big) \in T \mathcal{G}$ of the Euler
equation \eqref{E:EulerL} with $u(t)$ close to $u_{0}(t)$, let
\begin{equation}
\label{E:coord1}u = \Phi(t, w) \triangleq u_{0} (t) \circ\Psi\big(w(t)\big),
\quad w(t) \in B_{\delta_{0}}(W_{Euler}^{k, q}),
\end{equation}
where $\Psi$ is given in Proposition \ref{P:coordCG} which also implies
$\Phi(t, \cdot)$, for any $t \in\mathbb{R}$, is a local diffeomorphism from
$W_{Euler}^{k, q}$ to $\mathcal{G}$. Therefore, the linearizations
\begin{align*}
&  \tilde X = \mathcal{D} \Phi(t,w)X= \big(D u_{0} \circ\Psi
(w)\big) \mathcal{D} \Psi(w) X, \quad &  &  X \in W_{Euler}^{k, q}\\
&  X = \big(\mathcal{D} \Phi(t, w)\big)^{-1} \tilde X = \big(\mathcal{D}
\Psi(w)\big)^{-1} \big((D u_{0}) ^{-1} \circ\Psi(w) \big) \tilde X, \quad &
&  \tilde X \in T_{u} \mathcal{G}%
\end{align*}
are isomorphisms from between $W_{Euler}^{k, q}$ and $T_{u} \mathcal{G} $.
Substitute \eqref{E:coord1} in to \eqref{E:EulerL}, one may compute
\begin{align}
&  u_{t} = \Phi_{t}(t, w) + \mathcal{D} \Phi(t, w) w_{t} = u_{0t} \circ\Psi(w)
+ (D u_{0} \circ\Psi(w)) \mathcal{D} \Psi(w) w_{t}\label{E:coord2}\\
&  v = u_{t} \circ u^{-1} = v_{0} + \left(  (D u_{0}) \circ u_{0}^{-1}\right)
\left(  \left(  \mathcal{D} \Psi(w) w_{t}\right)  \circ\Psi(w)^{-1} \circ
u_{0}^{-1}\right)  \label{E:coord3}%
\end{align}
and
\[%
\begin{split}
u_{tt} =  &  \Phi_{tt} (t, w) + 2 \mathcal{D} \Phi_{t} (t, w) w_{t} +
\mathcal{D} \Phi(t, w) w_{tt} + \mathcal{D}^{2} \Phi(t, w) (w_{t}, w_{t})\\
=  &  u_{0tt} \circ\Psi(w) + 2 \left(  D u_{0t} \circ\Psi(w)\right)
\mathcal{D} \Psi(w) w_{t} + \left(  D^{2} u_{0} \circ\Psi(w)\right)
\big(\mathcal{D} \Psi(w) w_{t}, \mathcal{D} \Psi(w) w_{t}\big)\\
&  + \left(  D u_{0} \circ\Psi(w)\right)  \left(  \mathcal{D}^{2} \Psi(w)
(w_{t}, w_{t}) + \mathcal{D} \Psi(w) w_{tt}\right)  .
\end{split}
\]
Therefore, for $u(t)$ close to $u_{0}(t)$, the Euler equation \eqref{E:EulerL}
is rewritten as
\begin{equation}
\label{E:EulerA}w_{tt} + \mathcal{F}(t, w, w_{t}) =0, \quad w\in B_{\delta
_{0}} (W_{Euler}^{k, q}), \;w_{t} \in W_{Euler}^{k, q},
\end{equation}
where, for $w \in B_{\delta_{0}} (W_{Euler}^{k, q})$ and $X \in W_{Euler}^{k,
q}$, the term $\mathcal{F} (t, w, X)$ is arranged into the linear and
quadratic parts in $X$
\begin{equation}
\label{E:CF}\mathcal{F} (t, w, X) = A(t, w) X + B(t, w) (X, X) \in
W_{Euler}^{k, q}%
\end{equation}
with the linear and bilinear (in $X$) operators $A(t, w)$ and $B(t, w)$ are
given by
\begin{align}
A(t, w) X =  &  2 \big(\mathcal{D} \Phi(t, w) \big)^{-1} \big(\mathcal{D}
\Phi_{t} (t, w) X + \mathcal{P} (u, v_{0} \circ u, \tilde X) \big)\nonumber\\
=  &  2 \big(\mathcal{D} \Psi(w)\big)^{-1} \big((D u_{0}) ^{-1} \circ\Psi(w)
\big) \big( (D v_{0} \circ u) \tilde X + \mathcal{P} (u, v_{0} \circ u, \tilde
X) \big) \label{E:A1}%
\end{align}
\begin{align}
B(t, w)(X, X) =  &  \big(\mathcal{D} \Phi(t, w) \big)^{-1} \big(\mathcal{D}%
^{2} \Phi(t, w) (X, X) + \mathcal{P}(u, \tilde X, \tilde X) \big)\nonumber\\
=  &  \big(\mathcal{D} \Psi(w)\big)^{-1} \big( (D u_{0})^{-1} \circ
\Psi(w)\big) \Big(\big(D u_{0} \circ\Psi(w)\big)\mathcal{D}^{2} \Psi(w) (X,
X)\nonumber\\
&  \qquad\quad+ \big(D^{2} u_{0} \circ\Psi(w)\big) \big(\mathcal{D} \Psi(w) X,
\mathcal{D} \Psi(w) X\big) + \mathcal{P} (u, \tilde X, \tilde X) \Big),
\label{E:B}%
\end{align}
where
\[
u = \Phi(t, w) =u_{0} \circ\Psi(w) \qquad\tilde X = \mathcal{D} \Phi(t, w) X=
\big(D u_{0} \circ\Psi(w)\big) \mathcal{D} \Psi(w) X.
\]
In the above calculation, the invariance of $\mathcal{P}$ under the right
translation \eqref{E:RT} and equation \eqref{E:EulerL} were used in handling
both $u_{tt}$ and $u_{0tt}$. Here even though the linear operator
$\big(\mathcal{D} \Phi(t,w) \big)^{-1}$ acts only on the subspace
$T_{u}\mathcal{G}$, the terms $A(t, w)X$ and $B(t, w)(X, X)$ are well-defined
and thus $\mathcal{F} (t,w, X) \in W_{Euler}^{k, q}$. In fact,
\eqref{E:Hodge3} implies
\[
\mathcal{D} \Phi_{t} (t, w) X + \mathcal{P} (u, v_{0} \circ u, \tilde X)=(D
v_{0} \circ u) \tilde X + \mathcal{P} (u, v_{0} \circ u, \tilde X) \in T_{u}
\mathcal{G}%
\]
and thus $A(t, w)$ is well-defined. To see that $B(t, w)$ is well-defined, we
note the second linearization of $\Phi$ along a line $w + s X$, $s
\in\mathbb{R}$, at $s=0$, is given by
\[
u_{ss}= \mathcal{D}^{2} \Phi(t, w) (X, X) = \big(D u_{0} \circ\Psi
(w)\big)\mathcal{D}^{2} \Psi(w) (X, X) + \big(D^{2} u_{0} \circ\Psi
(w)\big) \big(\mathcal{D} \Psi(w) X, \mathcal{D} \Psi(w) X\big).
\]
Let
\[
Y = u_{s} \circ u^{-1} = \tilde X \circ u^{-1} \in W_{Euler}^{k, q},
\]
then
\[
u_{ss} = Y_{s} \circ u + DY (\tilde X).
\]
Since $Y_{s} \in W_{Euler}^{k, q}$ and \eqref{E:Hodge3} implies
\[
D Y(\tilde X) + \mathcal{P}(u, \tilde X, \tilde X) \in T_{u} \mathcal{G},
\]
we obtain
\[
\mathcal{D}^{2} \Phi(t, w) (X, X) + + \mathcal{P}(u, \tilde X, \tilde X) =
u_{ss} + \mathcal{P}(u, \tilde u_{s}, \tilde u_{s}) \in T_{u} \mathcal{G}%
\]
and thus
\[
B(t, w)(X, X) = \big(\mathcal{D} \Phi(t, w) \big)^{-1} \big(\mathcal{D}^{2}
\Phi(t, w) (X, X) + \mathcal{P}(u, \tilde X, \tilde X) \big) \in W_{Euler}^{k,
q}%
\]
is well-defined.

\begin{remark}
An alternative way to rewrite the Euler's equation to derive the above form is
to follow the Lagrangian variational principle. One may first express the
action $\int\int_{\Omega}\frac{|u_{t}|^{2}}2 dy dt$, defined on $T\mathcal{G}%
$, using \eqref{E:coord1} and \eqref{E:coord2}. The Euler's equation in terms
of $w$ follows from the the variation of the action.
\end{remark}

Recall we assumed in (A1) that $v_{0} \in W^{k+r, q}$ with $r\ge4$ and $\mu>
\mu_{0} \ge0$ is a constant fixed before \eqref{E:LagM0}.

\begin{lemma}
\label{L:CF} The nonlinear mapping $\mathcal{F}$ satisfies
\[
\mathcal{F} \in C^{r-4} \big(\mathbb{R} \times B_{\delta_{0}} (W_{Euler}^{k,
q}) \times W_{Euler}^{k, q}, W_{Euler}^{k, q}\big), \quad\mathcal{F}(t, w, 0)
\equiv0.
\]
Moreover, there exist $K>0$ depending only on $r$ and $k$ and $C>0$ depending
only on $r, n, k, q, v_{0}$, such that, for $t \in\mathbb{R}$, $w, X \in
W_{Euler}^{k, q}$ and $|w|_{W^{k, q}} < \delta_{0}$, we have
\begin{align*}
&  |A(t, \cdot)|_{C^{r-2} \big(B_{\delta_{0}} (W_{Euler}^{k, q}),
L(W_{Euler}^{k, q})\big)} \le C e^{K \mu|t|},\\
&  |B(t, \cdot)|_{C^{r-2} \big(B_{\delta_{0}} (W_{Euler}^{k, q}),
L(W_{Euler}^{k, q} \otimes W_{Euler}^{k, q}, W_{Euler}^{k, q})\big)} \le C
e^{K \mu|t|}.
\end{align*}
Here the highest order derivative is in the G\^ateaux sense which is
sufficient to yield the $C^{r-3,1}$ bounds.
\end{lemma}

\begin{proof}
The smoothness of $\mathcal{F}$ follows directly from its expression
\eqref{E:CF} -- \eqref{E:B} and Propositions \ref{P:coordCG} -- \ref{P:ODE}.
The property $\mathcal{F}(t, w, 0) \equiv0$ follows directly from \eqref{E:CF}
which is actually a consequence of the right translation invariance of the
Euler equation. To demonstrate the latter, we notice that $(u_{0} \circ\phi,
u_{0t} \circ\phi)$ is a solution of \eqref{E:EulerL} for any $\phi
\in\mathcal{G}$. Then \eqref{E:coord1} implies that, for any $w \in
B_{\delta_{0}}(W_{Euler}^{k, q})$, $(w, 0)$ is a time independent solution of
\eqref{E:EulerA} and thus $\mathcal{F}(t, w, 0)=0$. The derivation of the
estimates is tedious, but straightforwardly from Proposition \ref{P:coordCG}
-- \ref{P:CP}, \eqref{E:CP1}, and \eqref{E:CF} -- \eqref{E:B}.

\end{proof}

\begin{remark}
As proved in Proposition \ref{P:CP}, \eqref{E:EulerL} is a smooth infinite
dimensional ODE. However, the local coordinate systems based on the
composition would always cause loss of derivatives due to Proposition
\ref{P:translation}. Here our local coordinate mapping $\Phi(t, \cdot)$ allows
us to obtain some limited smoothness due to assumption (A1) of the extra
regularity of $v_{0}$.
\end{remark}

\subsection{Linear exponential dichotomy in Lagrangian coordinates.}

\label{SS:exp-di}

Since $\mathcal{F}(t, w, 0) =0$ for any small $w \in T_{id}\mathcal{G}$, we
have $\mathcal{D}_{w} F(t, w, 0)=0$. We can rewrite \eqref{E:EulerA} as
\begin{equation}
\label{E:EulerB}z_{t} = A_{0} (t) z + F(t, z), \qquad z= (z_{1}, z_{2})^{T}
\in B_{\delta_{0}} (W_{Euler}^{k, q}) \times W_{Euler}^{k, q}%
\end{equation}
where
\[
A_{0}(t) =
\begin{pmatrix}
0 & I\\
0 & -A(t, 0)
\end{pmatrix}
, \qquad F(t, z) =
\begin{pmatrix}
0\\
A(t,0) z_{2} -\mathcal{F}(t, z_{1}, z_{2})
\end{pmatrix}
.
\]
From \eqref{E:A1}, the explicit form of $A(t, 0)$ is given by
\[
A(t, 0) X = 2 \big(Du_{0}(t)\big)^{-1} \Big( \big(D v_{0} \circ u_{0}%
(t)\big) Du_{0}(t) X + \mathcal{P} \big(u_{0}, v_{0} \circ u_{0}(t), Du_{0}(t)
X\big) \Big)
\]
and Lemma \ref{L:CF} implies that there exists $C>0$ independent of
$\delta_{0} $ such that
\begin{equation}
\label{E:F0}F(t, 0) =0= \mathcal{D}_{z} F(t, 0), \quad|\mathcal{D}_{z}^{2} F
(t, \cdot)|_{C^{r-4} \big(B_{\delta_{0}} (W_{Euler}^{k, q})^{2}, W_{Euler}^{k,
q}\big))} \le C e^{K \mu|t|},
\end{equation}
where again the highest order derivative is in the G\^ateaux sense which is
sufficient to yield the $C^{r-3,1}$ bounds. The linearization of
\eqref{E:EulerA} takes the form of
\begin{equation}
\label{E:LEulerA}w_{tt} + \mathcal{D}_{X} \mathcal{F} (t, 0, 0) w_{t}= 0
\Leftrightarrow z_{t} = A_{0}(t) z
\end{equation}
whose well-posedness is guaranteed by Lemma \ref{L:CF}. Let $T(t, t_{0})$ be
the solution operator of \eqref{E:LEulerA} with initial time $t_{0}$ and
terminal time $t$.

On the one hand, for $w\in W_{Euler}^{k,q}$, $z=(w,0)^{T}$ is a solution of
\eqref{E:LEulerA}. On the other hand, linearizing \eqref{E:coord3} at the
steady solution $z=(w_{0},0)^{T}$ one may compute $z=(w(t),w_{t}(t))^{T}$ is a
solution of the linearization of \eqref{E:EulerB} at $z=(w_{0},0)^{T}$, where
\[
w_{t}(t)=\big(\mathcal{D}\Psi(w_{0})\big)^{-1}\big((Du_{0}(t))^{-1}\circ
\Psi(w_{0})\big)\big(v(t)\circ u_{0}(t)\circ\Psi(w_{0})\big)
\]
and $v=v(t)$ is a solution of \eqref{E:LEuler}. In particular, taking
$w_{0}=0$, we have that $z=(w(t),w_{t}(t))^{T}$ is a solution solution of
\eqref{E:LEulerA} where
\begin{equation}
w_{t}(t)=\big(Du_{0}(t)\big)^{-1}\big(v(t)\circ u_{0}(t)\big).
\label{formula-w-t}%
\end{equation}

This correspondence between the linearized solutions of \eqref{E:LEuler} and
\eqref{E:LEulerA} and assumption (A2) would yield the ($t$-dependent)
exponential dichotomy $(W_{Euler}^{k,q})^{2}=Y_{u}(t)\oplus Y_{cs}(t)$ of
\eqref{E:LEulerA} such that $T(t,t_{0})Y_{u,cs}(t_{0})=Y_{u,cs}(t)$ and the
exponential decay rate of $T(t,t_{0})$ as $t\rightarrow-\infty$ (or $+\infty$)
in $Y_{u}(t)$ (or $Y_{cs}(t)$) is bounded roughly by $\lambda_{u}$ (or
$\lambda_{cs}$). To define $Y_{u}(t)$, it is natural from (\ref{formula-w-t})
that, for $(w,w_{t})\in Y_{u}(t)$, $w_{t}$ takes the form of $\big(Du_{0}%
(t)\big)^{-1}\big(e^{tL}v\circ u_{0}(t)\big)$ with $v\in X_{u}$, where $L$ (as
well as $L_{u,cs}$) is the linear operator defined in the linearized Euler
equation \eqref{E:LEuler}. Also the decay of $w$ as $t\rightarrow-\infty$
requires it take the form of $w(t)=\int_{-\infty}^{t}w_{t}dt^{\prime}$.
Therefore, for any $t\in\mathbb{R}$, let
\[
Y_{u}(t)=\{\Big(\int_{-\infty}^{t}\big(Du_{0}(\tau)\big)^{-1}\big((e^{\tau
L}v)\circ u_{0}(\tau)\big)d\tau,\,\big(Du_{0}(t)\big)^{-1}\big(e^{tL}v\circ
u_{0}(t)\big)\Big)^{T}\mid v\in X_{u}\}.
\]
The convergence of the above infinite integral follows directly from
assumptions (A2) and (A3) with $K$ sufficient large depending only on $k$.
Similarly, for $(w,w_{t})\in Y_{cs}(t)$, $w_{t}$ takes the form of
$\big(Du_{0} (t)\big)^{-1}\big(e^{tL}v\circ u_{0}(t)\big)$ with $v\in X_{cs}$,
and $w$ should take the form of
\[
w_{0}+\int_{0}^{t}\big(Du_{0}(\tau)\big)^{-1}\big((e^{tL} v)\circ u_{0}%
(\tau)\big)d\tau.
\]
However, $w_{0} \in W_{Euler}^{k, q}$ is arbitrary and thus we can absorb the
integral term into $w_{0}$. Define
\begin{align*}
Y_{cs}(t)=\{\Big(w,\,\big(Du_{0}(t)\big)^{-1}  &  \big(e^{tL}v\circ
u_{0}(t)\big)\Big)^{T}\mid w\in W_{Euler}^{k,q},\;v\in X_{cs}\}.
\end{align*}

\begin{lemma}
\label{L:ED} It holds $(W_{Euler}^{k, q})^{2} = Y_{cs}(t) \oplus Y_{u}(t)$.
Moreover, let $P_{u, cs} (t) \in L\big((W_{Euler}^{k, q})^{2}\big)$ be the
projections associate to this decomposition and
\[
T_{u, cs} (t, t_{0})= T(t, t_{0})|_{Y_{u, cs} (t_{0})}.
\]
Then for any $t, t_{0}\le0$, we have $T_{u, cs}(t, t_{0}) Y_{u, cs} (t_{0})=
Y_{u, cs} (t)$ and there exist constants $K>0$ depending only on $k$ and
$C_{0}\ge1$ depending only on $k, n, q, v_{0}$ such that if $\lambda_{u} >
K\mu$, we have
\begin{align*}
&  |P_{u, cs}(t)|_{L\big((W_{Euler}^{k, q})^{2}\big)} \le C_{0} e^{-K \mu t},
\qquad\forall\; t\le0\\
&  |T_{cs}(t, t_{0})| \le C_{0} e^{\lambda_{cs} (t-t_{0}) - K \mu t_{0}},
\qquad\forall\; 0\ge t\ge t_{0}\\
&  |T_{u} (t, t_{0})| \le C_{0} e^{(\lambda_{u} - K \mu) (t-t_{0})- K \mu
t_{0} }, \qquad\forall\; t\le t_{0}\le0.
\end{align*}

\end{lemma}

\begin{remark}
This lemma shows that unlike the traditional exponential dichotomy, the norms
of the projections in the invariant splitting here is not uniformly bounded in
$t$ and may approach $\infty$ as $t \to-\infty$. This means that the angles
between the unstable and center-stable subspaces $Y_{u, cs}(t)$ of
\eqref{E:LEulerA} may not have a uniform positive lower bound as $t \to
-\infty$.
\end{remark}

\begin{proof}
We first show the invariance of $Y_{cs, u}(t)$ under $T(t, t_{0})$. In fact,
for any
\begin{align*}
Y_{u} (t_{0}) \ni z = (w, w_{1})^{T} =\Big(\int_{-\infty}^{t_{0}} \big(Du_{0}
(\tau)\big)^{-1} \big( (e^{\tau L} v) \circ u_{0}(\tau)\big) d\tau, \,
\big(Du_{0} (t_{0})\big)^{-1} \big( e^{t_{0} L} v \circ u_{0}(t_{0}%
)\big)\Big)^{T},
\end{align*}
where $v \in X_{u}$, let
\[
v(t) = e^{t L} v \in X_{u}, \qquad\tilde w= \int_{-\infty}^{0} \big(Du_{0}
(\tau)\big)^{-1} \big( (e^{\tau L} v) \circ u_{0}(\tau)\big) d\tau.
\]
Clearly,
\begin{align*}
Y_{u}(t) \ni z(t)=  &  \big(w(t), w_{1}(t)\big)^{T}\triangleq\Big(\int
_{-\infty}^{t} \big(Du_{0}(\tau)\big)^{-1} \big( (e^{\tau L} v) \circ
u_{0}(\tau)\big)d\tau, \, \big(Du_{0}(t)\big)^{-1} \big( e^{tL} v \circ
u_{0}(t)\big) \Big)^{T}\\
=  &  (\tilde w, 0)^{T} + \Big(\int_{0}^{t} \big(Du_{0}(\tau)\big)^{-1}
\big( v(\tau) \circ u_{0}(\tau)\big)d\tau, \, \big(Du_{0}(t)\big)^{-1}
\big( v(t) \circ u_{0}(t)\big) \Big)^{T}%
\end{align*}
is the solution of \eqref{E:LEulerA} with initial data $(\tilde w, v)$ at
$t=0$. Moreover, it satisfies $z(t_{0}) = z$ and thus $Y_{u}(t) \ni z(t) =
T(t, t_{0}) z$ which implies the invariance of $Y_{u}(t)$.

The above arguments also leads to the decay estimate of $T_{u}(t, t_{0})$. In
fact,one may compute
\[
w_{1} (t) = \big(Du_{0}(t)\big)^{-1} \Big( \big( e^{(t-t_{0})L} (( D
u_{0}(t_{0}) w_{1}) \circ(u_{0}(t_{0}))^{-1})\big) \circ u_{0}(t) \Big).
\]
Since $v \in X_{u}$, we obtain from \eqref{E:LagM0} and assumptions (A2) and
(A3)
\begin{align*}
|w_{1} (t)|_{W^{k, q}} \le C e^{\lambda_{u} (t- t_{0})- K \mu t} |w_{1}|_{W^{k
, q}}, \qquad\forall\; t\le t_{0}\le0.
\end{align*}
Finally from $w(t) = \int_{-\infty}^{t} w_{1}(\tau) d\tau$, we obtain the
estimate for $T_{u} (t, t_{0})$. The proof of the invariance of $Y_{cs}(t)$
and the estimate for $T_{cs} (t, t_{0})$ are similar.

To prove the direct sum and obtain the bounds on the projection operators, let
$P_{cs, u}^{0} \in L(W_{Euler}^{k, q})$ be the projections given by the
decomposition $W_{Euler}^{k, q} = X_{cs} \oplus X_{u}$ assumed in hypothesis
(A2). Given any $z=(w, w_{1})^{T} \in(W_{Euler}^{k, q})^{2}$ and $t \le0$,
let
\begin{align*}
&  v_{cs, u} (t) = P_{cs, u}^{0} e^{-tL}\Big( \big(Du_{0}(t) w_{1}\big) \circ
u_{0}(t)^{-1} \Big) \in X_{cs, u}\\
&  w_{0} (t) = w- \int_{-\infty}^{t} \big(Du_{0} (\tau)\big)^{-1}
\big( (e^{\tau L} v_{u} (t) ) \circ u_{0}(\tau)\big) d\tau
\end{align*}
From \eqref{E:LagM0} and assumptions (A2) and (A3), we have
\[
|e^{\tau L} v_{u}(t) |_{W^{k, q}} \le C e^{\lambda_{u} (\tau-t) - K \mu t}
|w_{1}|_{W^{k, q}}, \quad\forall\; \tau\le t.
\]
Let
\begin{align*}
z_{u} =  &  \Big(\int_{-\infty}^{t} \big(Du_{0} (\tau)\big)^{-1}
\big( (e^{\tau L} v_{u}(t)) \circ u_{0}(\tau)\big) d\tau, \, \big(Du_{0}
(t)\big)^{-1} \big( e^{tL} v_{u}(t) \circ u_{0}(t)\big)\Big)^{T} \in
Y_{u}(t)\\
z_{cs}=  &  \Big(w_{0}(t) , \, \big(Du_{0} (t)\big)^{-1} \big( e^{tL} v_{cs}
(t) \circ u_{0}(t)\big) \Big)^{T} \in Y_{cs}(t).
\end{align*}
Obvious $z= z_{u} + z_{cs}$ and this splitting is unique. Therefore
$(W_{Euler}^{k, q})^{2} = Y_{cs}(t) \oplus Y_{u}(t)$ and $P_{cs, u} (t) z =
z_{cs, u}$. It is straight forward to first obtain the estimates on $P_{u}
(t)$ based on the above inequalities and the bound on $P_{cs}(t)= I- P_{u}(t)$
also follows.
\end{proof}

Let
\[
F_{u, cs} (t, z)= P_{u, cs} (t) F (t, z).
\]
Then Lemma \ref{L:ED} and \eqref{E:F0} imply that there exist $K>0$ and
$C_{1}>0$ such that for any $t\le0$, it holds
\begin{equation}
\label{E:F}F_{cs, u} (t, 0) =0= \mathcal{D}_{z} F_{cs,u}(t, 0), \quad
|\mathcal{D}_{z}^{2} F_{u,cs} (t, z)|_{C^{r-4} (B_{\delta_{0}} (W_{Euler}^{k,
q})^{2}, W_{Euler}^{k, q})} \le C_{1} e^{-K \mu t},
\end{equation}
where again the highest order derivative is in the G\^ateaux sense which is
sufficient to yield the $C^{r-3,1}$ bounds.

\subsection{Integral unstable manifolds of \eqref{E:EulerA}.}

\label{SS:t-InMa}

We will follow the Lyapunov-Perron integral equation method to construct the
integral unstable manifolds. In the standard construction of invariant
manifolds, where the bounds on the invariant splitting and the nonlinear terms
are uniform in $t$, small Lipschitz constant of $F_{cs, u}$ is sufficient in
the construction of local invariant manifolds. As we do not have $t$-uniform
estimates here, we repeatedly used the property that $F(t, z) = O(|z|^{2})$ to
yield an extra $|z(t)|= O(e^{\lambda t})$ with $\lambda t<0 $. This quadratic
nature of $F$ combined with the exponential gap condition \eqref{E:lambda}
allows us to complete the proof of Proposition \ref{P:LP1} in the below.
However, similar construction would not work for the construction of the
center-stable manifold since solutions $z(t)$ on the center-stable manifold do
not satisfy $z(t) \to0$ as $t\to\pm\infty$.

For $\lambda\in(\lambda_{cs}, \lambda_{u})$ and $\delta_{1}\le\delta_{0}$ to
be determined, let
\[
\Gamma= \{ z = (z_{u}, z_{cs}) \in C^{0}\big((-\infty, 0], \overline
{B_{\delta_{1}} (W_{Euler}^{k, q})^{2}} \big) \mid|z|_{\lambda}\le\delta
_{1}\}
\]
where $z_{u, cs} (t) \in Y_{u, cs} (t)$ and
\[
|z|_{\lambda}\triangleq\sup_{t\le0} e^{-\lambda t}|z(t)|_{W^{k, q}}.
\]
This $\Gamma$ is the set of functions with the desired backward in time decay
expected to be satisfied by the solutions on the unstable manifolds. For any
$z \in\Gamma$ and
\begin{equation}
\label{E:zu0}z_{u0} \in Y_{u}(0), \qquad|z_{u0}|_{W^{k, q}} \le\frac
{\delta_{1}}{2C_{0}}%
\end{equation}
where $C_{0}$ is given in Lemma \ref{L:ED}, define $\mathcal{L}(\cdot,
z_{u0})$, where $\tilde z = (\tilde z_{u}, \tilde z_{cs})= \mathcal{L}(z,
z_{u0})$ is defined as, for any $t\le0$,
\begin{equation}
\label{E:LP}%
\begin{cases}
\tilde z_{u} (t) = T_{u} (t, 0) z_{u0} + \int_{0}^{t} T_{u}(t, \tau) F_{u}
\big(\tau, z(\tau)\big) d\tau\\
\tilde z_{cs} (t) = \int_{-\infty}^{t} T_{cs} (t, \tau) F_{cs} \big(\tau,
z(\tau)\big) d\tau.
\end{cases}
\end{equation}

\begin{proposition}
\label{P:LP1} There exists $K>0$ depending only on $r$ and $k$ such that if
$\lambda$ and $\delta_{1}$ satisfy
\begin{align}
&  \lambda\in\big(\lambda_{cs} + K \mu, \, \lambda_{u} - K \mu
\big)\label{E:lambda}\\
&  C_{0}C_{1} \delta_{1} \big(\frac1{\lambda- \lambda_{cs}} + \frac
1{\lambda_{u} - K \mu- \lambda}\big) < \frac12 \label{E:delta1}%
\end{align}
where $C_{0}$ and $C_{1}$ are from Lemma \ref{L:ED} and \eqref{E:F},
respectively, then $\mathcal{L}(\cdot, z_{u0})$ is a contraction on $\Gamma$
with the Lipschitz constant $\frac12$ for any $z_{u0}$ satisfying
\eqref{E:zu0}. Moreover $|\mathcal{L}|_{C^{r-3,1}} \le C$ for some $C>0$
depending only on $r, k, q, n, $ and $v_{0}$.
\end{proposition}

\begin{remark}
Assumption (A3) with a reasonably large $K$, depending only on $r$ and $k$,
guarantees the existence of $\lambda$ and $\delta_{1}$ satisfying the above inequalities.
\end{remark}

\begin{remark}
It is standard in the invariant manifold theory (see, for example,
\cite{CL88}) to prove that $z\in\Gamma$ solves \eqref{E:EulerB} with $z_{u}(0)
= z_{u0}$ if and only if, $z$ is the fixed point of $\mathcal{L}(\cdot,
z_{u0}) $. Therefore, Proposition \ref{P:LP1} shows that, for any given
$z_{u0}$ satisfying \eqref{E:zu0}, there exists a unique solution of
\eqref{E:EulerB} satisfying the exponential decay as $t \to-\infty$ with the
decay rate at least $\lambda$.
\end{remark}

\begin{proof}
In the proof the generic constant $K$ may change from line to line, but always
depends only on $k$ and $r$. From the definition of $\mathcal{L}$, Lemma
\ref{L:ED}, assumptions (A3), \eqref{E:lambda}, and \eqref{E:zu0}, and the
second order Taylor expansion of $F$ based on \eqref{E:F}, (instead of the
usual small Lipschitz estimates of $F$), we obtain for $t\le0$
\begin{align*}
e^{-\lambda t} |\tilde z_{u}(t)|_{W^{k, q}} \le &  C_{0} |z_{u0}|_{W^{k, q}}+
\int_{t}^{0} \frac12 C_{0} C_{1} e^{ - \lambda t +(\lambda_{u} - K \mu)
(t-\tau) - K \mu\tau+ 2 \lambda\tau} d\tau|z|_{\lambda}^{2}\\
\le &  C_{0} |z_{u0}|_{W^{k, q}} + \frac12 C_{0} C_{1} |z|_{\lambda}^{2}
\int_{t}^{0}e^{(\lambda_{u} - K \mu- \lambda) (t-\tau) } d\tau\\
\le &  C_{0} |z_{u0}|_{W^{k, q}} + \frac{C_{0}C_{1} \delta_{1}}{2(\lambda_{u}
- K \mu- \lambda)} |z|_{\lambda}%
\end{align*}
and
\begin{align*}
e^{-\lambda t} |\tilde z_{cs}(t)|_{W^{k, q}} \le &  \int_{-\infty}^{t} \frac12
C_{0} C_{1} e^{ - \lambda t +\lambda_{cs} (t-\tau) - K \mu\tau+ 2 \lambda\tau}
d\tau|z|_{\lambda}^{2}\le\frac{C_{0}C_{1} \delta_{1} |z|_{\lambda}}{2(\lambda-
\lambda_{cs})}.
\end{align*}
Therefore \eqref{E:delta1} implies that
\begin{equation}
\label{E:CL}|\mathcal{L}(z, z_{u0})|_{\lambda}\le C_{0} |z_{u0}|_{W^{k, q}} +
\frac14 |z|_{\lambda}\le\frac34 \delta_{1}%
\end{equation}
and thus $\mathcal{L}(\cdot, z_{u0})$ maps $\Gamma$ into itself.

To prove $\mathcal{L}(\cdot, z_{u0})$ is a contraction on $\Gamma$, we note
that for any $z^{1,2} \in\Gamma$, \eqref{E:F} implies
\begin{align*}
&  |F_{u, cs}(t, z^{2}(t)) - F_{u, cs}(t, z^{1}(t))| \le C_{1} \delta_{1}
e^{(2\lambda- K\mu) t} |z^{2} - z^{1}|_{\lambda}, \quad\forall t\le0
\end{align*}
Therefore, we have
\begin{align*}
e^{-\lambda t} |\tilde z_{u}^{2} (t) - \tilde z_{u}^{1} (t)| \le &  C_{0}
C_{1} \delta_{1} \int_{t}^{0} e^{ (\lambda_{u} - K \mu-\lambda) (t-\tau)}
d\tau|z^{2} - z^{1}|_{\lambda}\le\frac{C_{0}C_{1}\delta_{1}|z^{2} -
z^{1}|_{\lambda}} {\lambda_{u} - K \mu-\lambda}\\
e^{-\lambda t} |\tilde z_{cs}^{2} (t) - \tilde z_{cs}^{1} (t)| \le &  C_{0}
C_{1} \delta_{1} \int_{-\infty}^{t} e^{ (\lambda_{cs} -\lambda) (t-\tau)}
d\tau|z^{2} - z^{1}|_{\lambda}\le\frac{C_{0}C_{1}\delta_{1}|z^{2} -
z^{1}|_{\lambda}} {\lambda-\lambda_{cs}}%
\end{align*}
and thus \eqref{E:delta1} implies that $\mathcal{L}(\cdot, z_{u0})$ is a contraction.

Since $\mathcal{L}$ is linear in $z_{u0}$, we only need to prove its
smoothness in $z \in\Gamma$. For any $z\in\Gamma$, formally the linearization
of $\mathcal{L}$ is given by
\[
\mathcal{D}_{z} \mathcal{L} (z, z_{u0}) z_{1} = \bar z_{1} = (\bar z_{1u},
\bar z_{1cs})
\]
where for $t\le0$,
\begin{equation}
\label{E:LCL}%
\begin{cases}
\bar z_{1u} (t) = \int_{0}^{t} T_{u}(t, \tau) \mathcal{D}_{z} F_{u} \big(\tau,
z(\tau)\big) z_{1}(\tau) d\tau\\
\bar z_{1cs} (t) = \int_{-\infty}^{t} T_{cs} (t, \tau) \mathcal{D}_{z} F_{cs}
\big(\tau, z(\tau)\big) z_{1} (\tau) d\tau.
\end{cases}
\end{equation}
The same procedure as in the above shows that $|\mathcal{D}_{z} \mathcal{L}%
|_{\lambda}\le\frac12$. To show it is indeed the derivative of $\mathcal{L}$,
take $z_{1,2} \in\Gamma$, let
\[
\tilde z_{1,2} = (\tilde z_{1,2u}, \tilde z_{1,2cs}) = \mathcal{L} (z_{1,2},
z_{u0}), \qquad\bar z = (\bar z_{u}, \bar z_{cs}) = \mathcal{D}_{z}
\mathcal{L} (z_{1}, z_{u0}) (z_{2} - z_{1}).
\]
It is straight forward to compute from \eqref{E:F} and Lemma \ref{L:ED}, that
for $t\le0$,
\begin{align*}
&  e^{-\lambda t}|\tilde z_{2u}(t) - \tilde z_{1u}(t) - \bar z_{u} (t) |\\
=  &  e^{-\lambda t} |\int_{0}^{t} T_{u}(t, \tau)\Big(F_{u} \big(\tau, z_{2}
(\tau) \big) - F_{u}\big(\tau, z_{1}(\tau)\big) - \mathcal{D} F_{u} \big(\tau,
z_{1}(\tau)\big) \big(z_{2} (\tau) - z_{1}(\tau) \big) \Big)d\tau|\\
\le &  C_{0}C_{1} |z_{2} - z_{1}|_{\lambda}^{2} \int_{t}^{0} e^{-\lambda t +
(\lambda_{u} - K \mu)(t-\tau) - K \mu\tau+ 2\lambda\tau} d\tau\le\frac
{C_{0}C_{1}|z^{2} - z^{1}|_{\lambda}^{2}} {\lambda_{u} - K \mu-\lambda}.
\end{align*}
The estimate for the center-stable component is very similar and this proves
that $\mathcal{D}_{z} \mathcal{L}$ is indeed the derivative of $\mathcal{L}$.

Finally, we will show that $\mathcal{D}_{z} \mathcal{L}$ is Lipschitz. In
fact, let
\[
\bar z_{1,2} = (\bar z_{12u}, \bar z_{12cs}) = \mathcal{D}_{z} \mathcal{L}%
(z_{1,2}, z_{u0}) z.
\]
Then for any $t\le0$,
\begin{align*}
&  e^{-\lambda t} |\bar z_{2cs} (t) - \bar z_{1cs} (t)| =e^{-\lambda t}
|\int_{-\infty}^{t} T_{cs} (t, \tau) \big(\mathcal{D}_{z} F_{cs} (\tau,
z_{2}(\tau)) - \mathcal{D}_{z} F(\tau, z_{1}(\tau))\big) z(\tau) d\tau|\\
\le &  C_{0}C_{1} |z_{2} - z_{1}|_{\lambda}|z|_{\lambda}\int_{-\infty}^{t}
e^{(\lambda_{cs} -\lambda) (t-\tau) + (\lambda- K \mu) \tau} d\tau\le
\frac{C_{0} C_{1}}{\lambda_{cs} -\lambda} |z_{2} - z_{1}|_{\lambda
}|z|_{\lambda}.
\end{align*}
The estimates for the unstable component is the same and the proof of the
higher order smoothness is similar.
\end{proof}

From the Contraction Mapping Theorem and Proposition \ref{P:LP1}, the mapping
$\mathcal{L}$ has a unique fixed point
\[
z_{\ast}(t,z_{u0})=\big(z_{\ast u}(t,z_{u0}),z_{\ast cs}(t,z_{u0})\big)
\]
which is $C^{r-3,1}$ in $z_{u0}$ and satisfies $z_{\ast u}(0,z_{u0})=z_{u0}$.
Moreover, \eqref{E:CL} implies
\begin{equation}
|z_{\ast}(u_{0})|_{\lambda}\leq2C_{0}|z_{u0}|_{W^{k,q}}. \label{E:z*}%
\end{equation}
Like in the standard invariant manifold theory, these are solutions on the
invariant integral unstable manifold of the non-autonomous system
\eqref{E:EulerA}. Define
\[
h_{L}(z_{u0})=z_{\ast cs}(0,z_{u0})\in X_{cs}%
\]
for all $z_{u0}$ satisfying \eqref{E:zu0}, then the $C^{r-3,1}$ graph
\[
W_{L}^{u}\triangleq graph(h_{L})
\]
defines the slice of the unstable integral manifold of \eqref{E:EulerA} for
$t_{0}=0$. Obviously the uniqueness of the fixed point implies that $z_{\ast
}(\cdot,0)=0$ and thus $h_{L}(0)=0$ and $0\in W_{L}^{u}$. Differentiating the
fixed point equation we obtain
\[
\mathcal{D}_{z_{u0}}z_{\ast}(z_{u0})=\mathcal{D}_{z_{u0}}\mathcal{L}%
\big(z_{\ast}(z_{u0}),z_{u0}\big)+\mathcal{D}_{z}\mathcal{L}\big(z_{\ast
}(z_{u0}),z_{u0}\big)\mathcal{D}_{z_{u0}}z_{\ast}(z_{u0})
\]
Clearly, \eqref{E:LCL} implies that $\mathcal{D}_{z}\mathcal{L}(0,z_{u0})=0$
and thus
\[
\mathcal{D}_{z_{u0}}z_{\ast}(\cdot,0)=\mathcal{D}_{z_{u0}}\mathcal{L}%
(0,0)=\big(T_{u}(\cdot,0),0\big)
\]
which does not have the center-stable component. Therefore, we obtain that
\[
\mathcal{D}_{z_{u0}}h_{L}(0)=0
\]
which means that, at $0$, the tangent space of the unstable integral manifold
$W_{L}^{u}$
\[
T_{0}W_{L}^{u}=Y_{u}(0)=\{(U(v),\,v)^{T}\mid v\in X_{u}\},
\]
where
\begin{equation}
U(v)\triangleq\int_{-\infty}^{0}\big(Du_{0}(\tau)\big)^{-1}\big((e^{\tau
L}v)\circ u_{0}(\tau)\big)d\tau,\quad|U|_{L(X_{u},W_{Euler}^{k,q})}\leq\infty.
\label{E:U}%
\end{equation}
Here the boundedness of $U$ follows from \eqref{E:LagM0}, \eqref{E:lambda} and
assumptions (A2) and (A3).

\subsection{Unstable manifold in the Eulerian coordinates.}

\label{SS:InMa}

From the unstable integral manifold (at $t=0$) $W_{L}^{u}$ constructed in the
Lagrangian coordinates and the corresponding relationship given in
\eqref{E:coord1} and \eqref{E:coord3}, we obtain the $C^{r-3,1}$ invariant
unstable manifold in the Eulerian coordinates
\[
W^{u} \triangleq\{ v =v_{0} + \big(\mathcal{D}\Psi(w) w_{1}\big) \circ
\Psi(w)^{-1} \mid(w, w_{1}) \in W_{L}^{u}\}.
\]
The above expression was derived by substituting $t=0$ into \eqref{E:coord3}.

\begin{lemma}
\label{L:UM} There exist $K>0$ $\left(  \text{depending only on }%
r\ \text{and\ }k\right)  $, $\delta_{2},\ C>0$ depending only on $r,k,q,n$,
and $v_{0}$ such that

\begin{enumerate}
\item There exists $H\in C^{r-3,1}\big(B_{3\delta_{2}}(X_{u}),X_{cs}\big)$
satisfying $|H|_{C^{r-3,1}}\leq C$, $H(0)=0$, $DH(0)=0$, and
\[
\{v_{0}+v_{1}+H(v_{1})\mid v_{1}\in B_{\delta_{2}}(X_{u})\}\subset W^{u}%
\cap\big(v_{0}+B_{2\delta_{2}}(W_{Euler}^{k,q})\big)\subset\{v_{0}%
+v_{1}+H(v_{1})\mid v_{1}\in B_{3\delta_{2}}(X_{u})\}.
\]

\item For any $v_{\#}\in\big(v_{0}+B_{2\delta_{2}}(W_{Euler}^{k,q})\big)\cap
W^{u}$ the solution $v(t)$ of the Euler equation (E) with the initial value
$v(0)=v_{\#}$ satisfies
\[
v(t)\in W^{u},\quad|v(t)-v_{0}|_{W^{k,q}}\leq C|v_{\#}-v_{0}|_{W^{k,q}%
}e^{(\lambda-K\mu)t},\qquad\forall\;t\leq0.
\]

\end{enumerate}
\end{lemma}

\begin{proof}
(1) For any $v_{1}\in X_{u}$ with
\[
|v_{1}|_{W^{k,q}}<\frac{\delta_{1}}{2C_{0}\big(1+|U|_{L(X_{u},W_{Euler}%
^{k,q})}\big)},
\]
let $z_{u0}=(Uv_{1},v_{1})\in Y_{u}(0)$, where $U$ is defined in \eqref{E:U}
and it implies $z_{u0}$ satisfies \eqref{E:zu0}, and
\[
G(v_{1})=\big(\mathcal{D}\Psi(w)w_{1}\big)\circ\Psi(w)^{-1}\quad\text{ where
}(w,w_{1})=z_{\ast}(0,z_{u0})=z_{u0}+h_{L}(z_{u0})\in W_{L}^{u}.
\]
Clearly the definition of $W_{L}^{u}$ and $W^{u}$ and the properties of
$h_{L}$ imply $W^{u}=\{v_{0}+G(v_{1})\}$ and
\[
G(0)=0,\quad DG(0)v_{1}=w_{1}=v_{1},
\]
and $G\in C^{r-3,1}$ with bounds depending only on $r,k,q,n$, and $v_{0}$.
Therefore, the existence and properties of $H$ follows from the Implicit
Function Theorem immediately.

(2) For any
\[
v_{\#}=v_{0}+\big(\mathcal{D}\Psi(w_{\#})w_{1\#}\big)\circ\Psi(w_{\#}%
)\in\big(v_{0}+B_{2\delta_{2}}(W_{Euler}^{k,q})\big)\cap W^{u},\text{ with
}(w_{\#},w_{1\#})\in W_{L}^{u},
\]
let $v(t)$ be the solution of (E) with $v(0)=v_{\#}$ and $z(t)=\big(w(t),w_{t}%
(t)\big)$ the solution of \eqref{E:EulerA} with the initial value
$z(0)=(w_{\#},w_{1\#})$. Since $z(0)=(w_{\#},w_{1\#})\in W_{L}^{u}$ and thus
$z\in\Gamma$ and \eqref{E:z*} implies
\[
|w(t)|_{W^{k,q}}+|w_{t}(t)|_{W^{k,q}}\leq e^{\lambda t}|z|_{\lambda}\leq
2C_{0}e^{\lambda t}|z_{u0}|_{W^{k,q}},\qquad t\leq0.
\]
Therefore \eqref{E:coord3} implies the desired decay estimate.
\[
|v(t)-v_{0}|_{W^{k,q}}\leq Ce^{(\lambda-K\mu)t}\rightarrow0\text{ as
}t\rightarrow-\infty.
\]

To see the local invariance of $W^{u}$, fixed $T\geq0,\ $for $t_{0}\in
(-\infty,T]$, let
\[
\tilde{w}(t)=\Psi^{-1}\Big(u_{0}(t_{0})\circ\Psi\big(w(t+t_{0})\big)\circ
u_{0}(t_{0})^{-1}\Big),\quad t\leq0.
\]
Since the right translation invariance and \eqref{E:coord1} imply
\[
\tilde{u}(t)=u_{0}(t)\circ\Psi\big(\tilde{w}(t)\big)=u_{0}(t+t_{0})\circ
\Psi\big(w(t+t_{0})\big)\circ u_{0}(t_{0})^{-1}=u(t+t_{0})\circ u_{0}%
(t_{0})^{-1}%
\]
is a solution of \eqref{E:EulerL}, we have $(\tilde{w},\tilde{w}_{t})$ is a
solution of \eqref{E:EulerA} which clearly corresponds to the solution
$\tilde{v}(t)=v(t+t_{0})$ of the Euler equation (E). Moreover, Proposition
\ref{P:coordCG} implies
\begin{align*}
|\tilde{w}(t)|_{W^{k,q}}+|\tilde{w}_{t}(t)|_{W^{k,q}}  &  \leq
C\big(|w(t+t_{0})|_{W^{k,q}}+|w_{t}(t+t_{0})|_{W^{k,q}}\big)\leq
Ce^{\lambda\left(  t+t_{0}\right)  }|z_{u0}|_{W^{k,q}}\\
&  \leq Ce^{\lambda T}\big(1+|U|_{L(X_{u},W_{Euler}^{k,q})}\big)\left\vert
v_{\#}-v_{0}\right\vert _{W^{k,q}}e^{\lambda t}\\
&  \leq2\delta_{2}Ce^{\lambda T}\big(1+|U|_{L(X_{u},W_{Euler}^{k,q}%
)}\big)e^{\lambda t}.
\end{align*}
By choosing
\[
\delta_{2}\leq\frac{\delta_{1}}{2Ce^{\lambda T}\big(1+|U|_{L(X_{u}%
,W_{Euler}^{k,q})}\big)},
\]
we obtain that the solution $\tilde{z}(t)=\big(\tilde{w}(t),\tilde{w}%
_{t}(t)\big)\in\Gamma$. Therefore $v(t_{0})=\tilde{v}(0)\in W^{u}$ which
implies the invariance.
\end{proof}

The property $DH(0)=0$ immediately implies, as expected, the tangent space at
the steady state $v_{0}$ is given by
\[
T_{v_{0}} W^{u} = X^{u}%
\]
and the proof of Theorem \ref{T:UnstableM} is complete.

\section{Two-dimensional Euler equations}

\label{S:2d}

In this and the next sections, we will illustrate how assumptions (A1) -- (A3)
can be satisfied for certain steady states of the Euler equation (E). In this
section, we consider the case of $\Omega=S^{1}\times(-y_{0},y_{0})$, that is,
$2\pi-$periodic in $x$ and with rigid walls on $\left\{  y=\pm y_{0}\right\}
$.

Let $v=(v_{1},v_{2})^{T}:\Omega\rightarrow\mathbb{R}^{2}$ satisfy $\nabla\cdot
v=0$ in $\Omega$ and $v\cdot N=0$ on $\partial\Omega$. On the one hand, let
\[
\omega=\partial_{x}v_{2}-\partial_{y}v_{1},\qquad s=\frac{1}{|\Omega|}%
\int_{\Omega}v_{1}dxdy
\]
be the curl and the average horizontal momentum, respectively. We note that
$s$ is an invariant of (E) due to the translation symmetry in $x$. On the
other hand, $v$ is uniquely determined by $\omega$ and $s$ through
\begin{equation}
v=J\nabla\Delta^{-1}\omega+se_{1},\qquad J=%
\begin{pmatrix}
0 & -1\\
1 & 0
\end{pmatrix}
,\quad e_{1}=(1,0)^{T} \label{E:v}%
\end{equation}
where $\Delta^{-1}$ is the inverse of the Laplacian with zero Dirichlet
boundary condition. It is clear that
\[
\frac{1}{C}|v|_{W^{k,q}}\leq|\omega|_{W^{k-1,q}}+|s|\leq C|v|_{W^{k,q}},\qquad
k\geq1,\;q>1
\]
for some $C>0$. In the $(\omega,s)$ representation, (E) takes the form
\begin{equation}
\omega_{t}+v\cdot\nabla\omega=0,\qquad s_{t}=0 \label{E:omega}%
\end{equation}
where $v$ is considered as determined by $(\omega,s)$ by \eqref{E:v}.

\begin{remark}
Due to the nontrivial first cohomology group of $\Omega$, the vorticity alone
does not determine a vector field in $W_{Euler}^{k,q}$ and thus the average
horizontal momentum has to be included in the reformulation of the problem. If
one considers $\Omega=T^{2}$, the $2$D torus, then both momentum invariants
each of which corresponds to a nontrivial element in the first cohomology
groups should be included.

\end{remark}

Suppose $v_{0}\in W^{k+4,q}$, $k>1+\frac{2}{q}$, is a steady state of (E),
which corresponds to $(\omega_{0},s_{0})$ has Lyapunov exponent $\mu_{0}\geq0$
(both forward and backward in time). We linearize \eqref{E:omega} at
$(\omega_{0},s_{0})$ to obtain
\begin{equation}%
\begin{pmatrix}
\omega\\
s
\end{pmatrix}
_{t}=-%
\begin{pmatrix}
v_{0}\cdot\nabla\omega\\
0
\end{pmatrix}
-%
\begin{pmatrix}
(J\nabla\Delta^{-1}\omega+se_{1})\cdot\nabla\omega_{0}\\
0
\end{pmatrix}
\triangleq L_{0}%
\begin{pmatrix}
\omega\\
s
\end{pmatrix}
+L_{1}%
\begin{pmatrix}
\omega\\
s
\end{pmatrix}
. \label{E:LOmega}%
\end{equation}
Assume that there exists an unstable eigenvalue $\lambda_{0}$ with
\ $\operatorname{Re}\lambda_{0}>\left(  k-1\right)  \mu_{0}$ of the linearized
Euler operator $L$ (defined in \eqref{E:LEuler}) on $L^{q}$, and let $v\in
L^{q}$ be the eigenfunction with corresponding $\omega$ and $s$. Then
obviously $s=0$ and
\[
\lambda_{0}\omega+v_{0}\cdot\bigtriangledown\omega=-v\cdot\nabla\omega_{0}.
\]
An integration of above along the steady trajectory$\ \mathbf{X}_{0}\left(
s\right)  $ yields
\[
\omega=\int_{0}^{\infty}e^{-\lambda_{0}s}v\cdot\nabla\omega_{0}\left(
\mathbf{X}_{0}\left(  s\right)  \right)  ds.
\]
By the standard bootstrap argument and the assumption $\operatorname{Re}%
\lambda_{0}>\left(  k-1\right)  \mu_{0}$, we get $\omega\in$ $W^{k-1,q}$ and
$v\in W^{k,q}$.

For any $\lambda_{-}\in\left(  (k-1)\mu_{0},\operatorname{Re}\lambda
_{0}\right)  $ which does not equal the real part of any unstable eigenvalue ,
let
\[
\tilde{X}_{cs}=\{(\omega,s)^{T}\mid\omega\in W^{k-1,q}\mid\lim\sup\frac{1}%
{t}\log|e^{t(L_{0}+L_{1})}\omega|_{W^{k-1,q}}\leq\lambda_{-}\}
\]
which is clearly a invariant subspace of $e^{L_{0}+L_{1}}$. Let
\[
\sigma_{cs}=\sigma(e^{L_{0}+L_{1}}|_{\tilde{X}_{cs}}),\qquad\sigma_{u}%
=\sigma(e^{L_{0}+L_{1}}|_{W^{k-1,q}\times\mathbb{R}})\backslash\sigma
_{cs},\qquad\lambda_{+}=\log(\inf\{|\lambda|\mid\lambda\in\sigma_{u}%
\})\geq\lambda_{-}.
\]
As the groups of bounded operators $e^{t(L_{0}+L_{1})}$ and $e^{tL}$\ are
conjugate through \eqref{E:v}, we also have the invariance of $X_{cs,u}$ under
$e^{tL}$. Since $e^{tL_{0}}\omega=\omega\circ u_{0}(t)^{-1}$, inequality
\eqref{E:LagM0} implies $|e^{tL_{0}}|_{L(W^{k-1,q})}\leq Ce^{(k-1)\mu|t|}$ for
any $\mu>\mu_{0}$ and some $C>0$ depending on $\mu$. In particular, $v_{0}$ is
divergence free, yields that $e^{tL_{0}}$ is a group of isometries on any
$L^{q}$ space. Since $L_{1}$ is a compact operator acting on $(\omega,s)^{T}$,
$e^{t(L_{0}+L_{1})}$ is a compact perturbation to $e^{tL_{0}}$ in the space
$W^{k-1,q}\ $and thus

\begin{itemize}
\item $\lambda_{+}>\lambda_{-}$ and $\sigma_{u}$ is an isolated compact subset
of $\sigma(e^{L_{0}+L_{1}})$. Let $\tilde{X}_{u}$ be the eigenspace of
$e^{L_{0}+L_{1}}$ corresponding to $\sigma_{u}$, and
\[
X_{cs,u}=\{v\in W_{Euler}^{k,q}\mid(\omega,s)\in\tilde{X}_{cs,u}\}.
\]

\item $\tilde X_{u}$ and $X_{u}$ are finite dimensional and

\item (A2) is satisfied for any $\lambda_{cs}$ and $\lambda_{u}$ with
$\lambda_{-}< \lambda_{cs} < \lambda_{u} < \lambda_{+}$.
\end{itemize}

Assumption (A3) depends on the Lyapunov exponents of $v_{0}$. In particular,
if $v_{0}$ is a linearly unstable shear flow, $\mu_{0}=0$ and (A3) is also
satisfied. An example is $v_{0}=(\sin\beta y,0)$. By \cite[Theorem 1.2]%
{Lin03}, $v_{0}$ is linearly unstable when $\beta>1$ and $\left(  \frac{\pi
}{2y_{0}}\right)  ^{2}<\beta^{2}-1$.

\begin{remark}
Consider rotating flows $v_{0}=U\left(  r\right)  \vec{e}_{\theta}$ in an
annulus $\Omega=\left\{  a<r<b\right\}  $. Then by similar arguments as above,
assumptions (A2)-(A3) are satisfied as long as $v_{0}$ is linearly unstable.
\end{remark}

\section{Three-dimensional Euler equations}

In this section, we construct examples of 3D unstable steady flows for which
Theorem \ref{T:UnstableM} can be applied to get unstable (stable) manifolds.
Consider $\Omega=T^{3}$ to be a 3D torus with periods $L_{x},L_{y}$ and
$L_{z}$ in $x,y\ $and $z$ variables. For any profile $U\left(  y,z\right)
,\ $the 3D shear flow $\vec{u}_{0}=\left(  U\left(  y,z\right)  ,0,0\right)  $
is a steady solution of 3D Euler equation. We construct unstable 3D shears
satisfying assumptions (A1)-(A3) in several steps.

The linearized 3D Euler equation around a 3D shear $\left(  U\left(
y,z\right)  ,0,0\right)  $ is
\begin{equation}
\partial_{t}u+Uu_{x}+vU_{y}+wU_{z}=-P_{x}, \label{eqn-u-3d-shear}%
\end{equation}%
\begin{equation}
\partial_{t}v+Uv_{x}=-P_{y},\ \partial_{t}w+Uw_{x}=-P_{z},
\label{eqn-v,w-3d-shear}%
\end{equation}%
\begin{equation}
u_{x}+v_{y}+w_{z}=0, \label{eqn-div-3d-shear}%
\end{equation}
with periodic boundary conditions. There are almost no results about the the
linear instability of general 3D shears. So we construct unstable 3D shears
near unstable 2D shear flows $\left(  U_{0}\left(  y\right)  ,0,0\right)  $
where $U_{0}\left(  y\right)  $ is periodic with period $L_{y}$. First, we
give a sufficient condition for linear instability of 2D periodic shears,
which generalizes the result in \cite{Lin03} for shear flows in a channel with
rigid walls.

\begin{lemma}
\label{L-2d-shear}Consider a periodic shear profile $U(y)\in C^{2}\left(
0,L_{y}\right)  $ with only one inflection value $U_{s}$ and
\begin{equation}
K(y)=-\frac{U^{\prime\prime}(y)}{U(y)-U_{s}}>0. \label{Po}%
\end{equation}
Let $-\alpha_{\max}^{2}$ be the lowest eigenvalue of the Sturm-Liouville
operator
\begin{equation}
L\varphi=-\varphi^{\prime\prime}-K(y)\varphi\label{SLO}%
\end{equation}
with the periodic boundary conditions on $y\in\left[  0,L_{y}\right]  $. Then
the Rayleigh equation
\begin{equation}
U^{\prime\prime}\phi-(U-c)\left(  \phi^{\prime\prime}-\alpha^{2}\phi\right)
=0, \label{Rayleigh}%
\end{equation}
with periodic boundary conditions on $y\in\left[  0,L_{y}\right]  $ has
unstable eigenmodes ($\operatorname{Im}c>0$) for any $\alpha\in\left(
0,\alpha_{\max}\right)  $.
\end{lemma}

\begin{remark}
Under the assumptions in the above Lemma, the lowest eigenvalue of $L$ is
negative, since $\left(  L\left(  1\right)  ,1\right)  =-\int K(y)dy<0$. A
typical example satisfying (\ref{Po}) is $U_{0}\left(  y\right)  =\sin\left(
\frac{2\pi}{L_{y}}y\right)  $ for which $K\left(  y\right)  =\left(
\frac{2\pi}{L_{y}}\right)  ^{2}$.
\end{remark}

\begin{proof}
The proof is similar to the case of rigid walls (\cite{Lin03}), so we only
point out some small modifications. Let $\phi_{s}$ be the eigenfunction of
$L\ $corresponding to the lowest eigenvalue $-\alpha_{\max}^{2}$. Then
$\left(  \phi,c,\alpha\right)  =\left(  \phi_{s},U_{s},\alpha_{\max}\right)  $
is a neutral solution to the Rayleigh equation (\ref{Rayleigh}). By
Sturm-Liouville theory, $-\alpha_{\max}^{2}$ is a simple eigenvalue and we can
take $\phi_{s}>0$. First, we study bifurcation of unstable modes near the
neutral mode. Denote $y_{1}$ to be a minimum point of $\phi_{s}$ and let
$y_{2}=y_{1}+L_{y}$. We normalize $\phi_{s}$ such that $\phi_{s}\left(
y_{1}\right)  =1,\ \phi_{s}^{\prime}\left(  y_{1}\right)  =0$. Define
$\phi_{1}\left(  y;c,\varepsilon\right)  $ and $\phi_{2}\left(
y;c,\varepsilon\right)  $ to be the solutions of
\begin{equation}
-\phi^{\prime\prime}+\frac{U^{\prime\prime}}{U-U_{s}-c}\phi+\left(
\alpha_{\max}^{2}+\varepsilon\right)  \phi=0,\text{ } \label{homo-ode}%
\end{equation}
with $\phi_{1}\left(  y_{1}\right)  =1,$ $\phi_{1}^{\prime}\left(
y_{1}\right)  =0$ and $\phi_{2}\left(  y_{1}\right)  =0,\phi_{2}^{\prime
}\left(  y_{1}\right)  =1.$ Here $\varepsilon<0$ and $\operatorname{Im}c>0.$
Define
\[
\ I\left(  c,\varepsilon\right)  =\phi_{1}\left(  y_{2};c,\varepsilon\right)
+\phi_{2}^{\prime}\left(  y_{2};c,\varepsilon\right)  -2,
\]
then the existence of a solution to the Rayleigh equation (\ref{homo-ode})
with periodic boundary conditions on $y\in\left[  y_{1},y_{2}\right]  \ $is
equivalent to the existence of a root of $I$ with $\operatorname{Im}c>0.$ When
$c\rightarrow0,$ $\varepsilon\rightarrow0-$ and $\left\vert \operatorname{Re}%
c\right\vert /\operatorname{Im}c$ remains bounded, as in \cite{Lin03} we can
show that $\phi_{1}\left(  y;c,\varepsilon\right)  $ $(\phi_{2}\left(
y;c,\varepsilon\right)  )\ $converges to $\phi_{s}\left(  y\right)  $
$(\phi_{z}\left(  y\right)  )\ $uniformly in $C^{1}[y_{1},y_{2}]$. Here,
$\phi_{z}\left(  y\right)  \in C^{1}\left[  y_{1},y_{2}\right]  $ satisfies
that $\phi_{z}^{\prime}\left(  y_{2}\right)  =1$ and $\phi_{z}\left(
y_{2}\right)  \neq0$ since $\phi_{z}\left(  y\right)  $ can not be another
eigenfunction associated with the simple eigenvalue $-\alpha_{\max}^{2}$. By
similar calculations as in \cite{Lin03}, it can be shown that when
$c\rightarrow0,$ $\varepsilon\rightarrow0-$ and $\left\vert \operatorname{Re}%
c\right\vert /\operatorname{Im}c$ remains bounded,%
\[
\frac{\partial I}{\partial\varepsilon}\rightarrow\phi_{z}\left(  y_{2}\right)
\int_{y_{1}}^{y_{2}}\phi_{s}^{2}\left(  y\right)  dy,
\]
and
\[
\frac{\partial I}{\partial c}\rightarrow-\phi_{z}\left(  y_{2}\right)  \left(
i\pi\sum_{k=1}^{l}\left(  \left\vert U^{\prime}\right\vert ^{-1}K\phi_{s}%
^{2}\right)  |_{y=a_{k}}+\mathcal{P}\int_{y_{1}}^{y_{2}}\left(  K\left(
y\right)  \phi_{s}^{2}\left(  y\right)  \right)  /\left(  U\left(  y\right)
-U_{s}\right)  dy\right)  \text{. }%
\]
Here, $a_{1},\cdots,a_{l}$ are the inflection points such that $U\left(
a_{k}\right)  =U_{s}$, $k=1,\cdots,l$ and $\mathcal{P}\int_{y_{1}}^{y_{2}}$
denotes the Cauchy principal part. Then by a variant of implicit function
theorem as in \cite{Lin03}, there exists $\varepsilon_{0}<0$ such that for any
$\varepsilon_{0}<\varepsilon<0$ , there is an unstable solution $\phi
_{\varepsilon}$ with $c=c\left(  \varepsilon\right)  \ $to Rayleigh's equation
(\ref{homo-ode}). By the same arguments in \cite{Lin03}, such unstable modes
can be continuated to all wave numbers $\alpha\in\left(  0,\alpha_{\max
}\right)  $.
\end{proof}

Our second step is to show that 3D shears near an unstable 2D shear are also
linearly unstable. More precisely, we have

\begin{lemma}
\label{L-3d-shear}Let $U_{0}\left(  y\right)  \in C^{2}\left(  0,L_{y}\right)
$ be such that the Rayleigh equation (\ref{Rayleigh}) has an unstable solution
with $\left(  \alpha_{0},c_{0}\right)  $ $\left(  \alpha_{0},\operatorname{Im}%
c_{0}>0\right)  $. Fixed $L_{z}>0$, consider $U\left(  y,z\right)  \in
C^{1}\left(  \left(  0,L_{y}\right)  \times\left(  0,L_{z}\right)  \right)  $
which is $L_{y},\ L_{z}$-periodic in $y$ and $z$ respectively. If $\left\Vert
U\left(  y,z\right)  -U_{0}\left(  y\right)  \right\Vert _{W^{1,p}\left(
\left(  0,L_{y}\right)  \times\left(  0,L_{z}\right)  \right)  }$ $\left(
p>2\right)  \ $is small enough, then there exists an unstable solution
$e^{i\alpha_{0}\left(  x-ct\right)  }\left(  u,v,w,P\right)  \left(
y,z\right)  $ to the linearized equation (\ref{eqn-u-3d-shear}%
)-(\ref{eqn-div-3d-shear}) with $\left\vert c-c_{0}\right\vert $ small.
Moreover, if $U\left(  y,z\right)  \in C^{\infty}$, then $\left(
u,v,w,P\right)  \in C^{\infty}$.
\end{lemma}

The proof of above lemma is almost the same as in the case of rigid walls
(\cite{lin-li}), so we skip it here. By Lemmas \ref{L-2d-shear} and
\ref{L-3d-shear}, there exist linearly unstable 3D shears $\vec{u}_{0}=\left(
U\left(  y,z\right)  ,0,0\right)  $. Below, we show that the assumption (A2)
of linear exponential dichotomy is true in spaces $W_{Euler}^{m,2}=H^{m}$
($m\geq1$ is integer)$\ $for such unstable 3D shears. Then the assumption (A3)
is automatic since $\mu_{0}=0$. Let $G_{t}=e^{Lt}$ be the linearized Euler
semigroup near a steady flow $\vec{u}_{0}\left(  \vec{x}\right)  $ and denote
$r_{ess}\left(  G_{t};H^{m}\right)  $ to be the essential spectrum radius of
$G_{t}$ in space $H^{m}$. By rather standard semigroup theory (see e.g.
\cite[Section 1 ]{shizuta}), to get the linear exponential dichotomy (A2), it
suffices to show that $r_{ess}\left(  G_{t};H^{m}\right)  =1$. The essential
spectrum of linearized Euler operator had been studied a lot (\cite{fv-91}
\cite{lif-mai} \cite{latushkin-shvydkoy-jfa} \cite{vishik96}) by using the
geometric optics method. We use the following characterization of
$r_{ess}\left(  G_{t};H^{m}\right)  $ in \cite{latushkin-shvydkoy-jfa}.

\begin{lemma}
\cite{latushkin-shvydkoy-jfa}Consider the following ODE system%
\begin{equation}
\left\{
\begin{array}
[c]{c}%
\vec{x}_{t}=\vec{u}_{0}\left(  \vec{x}\right) \\
\vec{\xi}_{t}=-\partial\vec{u}_{0}\left(  \vec{x}\right)  ^{T}\ \vec{\xi}\\
\vec{b}_{t}=-\partial\vec{u}_{0}\left(  \vec{x}\right)  \vec{b}+2\left(
\partial\vec{u}_{0}\left(  \vec{x}\right)  \vec{b},\vec{\xi}\right)  \vec{\xi
}\left\vert \vec{\xi}\right\vert ^{-2},
\end{array}
\right.  \label{ODE-ess-euler}%
\end{equation}
where$\ \vec{u}_{0}\left(  \vec{x}\right)  $ is a steady flow of 3D Euler
equation in $T^{3}$ and $\vec{x}\in T^{3},\vec{\xi},\vec{b}\in\mathbf{R}^{3}
$. Denote
\begin{equation}
\Lambda_{m}=\lim_{t\rightarrow\infty}\frac{1}{t}\ln\sup_{\substack{\vec{x}%
_{0}\in T^{3},\ \left\vert \vec{\xi}_{0}\right\vert =1 \\\vec{b}%
_{0}{\huge \perp}\vec{\xi}_{0},\ \left\vert \vec{b}_{0}\right\vert
=1}}\left\vert \vec{b}\left(  t\right)  \right\vert \left\vert \vec{\xi
}\left(  t\right)  \right\vert ^{m} \label{defn-lambda-m}%
\end{equation}
where $\left(  \vec{x}\left(  t\right)  ,\vec{b}\left(  t\right)  ,\vec{\xi
}\left(  t\right)  \right)  $ is the solution of (\ref{ODE-ess-euler}) with
initial data $\left(  \vec{x}_{0},\vec{\xi}_{0},\vec{b}_{0}\right)  $. Then we
have%
\[
r_{ess}\left(  G_{t};H^{m}\right)  =e^{t\Lambda_{m}}.
\]

\end{lemma}

\begin{lemma}
For $\vec{u}_{0}=\left(  U\left(  y,z\right)  ,0,0\right)  $ in $T^{3}$, we
have $\Lambda_{m}=0$.
\end{lemma}

\begin{proof}
Denote%
\[
\vec{x}\left(  t\right)  =\left(  x\left(  t\right)  ,y\left(  t\right)
,z\left(  t\right)  \right)  ,\ \vec{\xi}\left(  t\right)  =\left(  \xi
_{1}\left(  t\right)  ,\xi_{2}\left(  t\right)  ,\xi_{3}\left(  t\right)
\right)  ,\ \vec{b}\left(  t\right)  =\left(  b_{1}\left(  t\right)
,b_{2}\left(  t\right)  ,b_{3}\left(  t\right)  \right)  ,
\]
and%
\[
\vec{x}_{0}=\left(  x_{0},y_{0},z_{0}\right)  ,\ \vec{\xi}_{0}=\left(  \xi
_{1}^{0},\xi_{2}^{0},\xi_{3}^{0}\right)  ,\ \vec{b}_{0}=\left(  b_{1}%
^{0},b_{2}^{0},b_{3}^{0}\right)  .
\]
The solution of first two equations of (\ref{ODE-ess-euler}) yield
\[
x\left(  t\right)  =x_{0}+U\left(  y_{0},z_{0}\right)  t,\ y\left(  t\right)
=y_{0},\ z\left(  t\right)  =z_{0}%
\]
and
\[
\xi_{1}\left(  t\right)  =\xi_{1}^{0},\ \ \xi_{2}\left(  t\right)  =-U_{y}%
\xi_{1}^{0}t+\xi_{2}^{0},\ \xi_{3}\left(  t\right)  =-U_{z}\xi_{1}^{0}%
t+\xi_{3}^{0}.
\]
Plugging above forms into the equation of $\vec{b}\left(  t\right)  $, we
have
\begin{equation}
\dot{b}_{1}=-\left(  U_{y}b_{2}+U_{z}b_{3}\right)  +\frac{2\left(  \xi_{1}%
^{0}\right)  ^{2}\left(  U_{y}b_{2}+U_{z}b_{3}\right)  }{\left\vert \vec{\xi
}\left(  t\right)  \right\vert ^{2}}, \label{eqn-b-1}%
\end{equation}%
\begin{equation}
\dot{b}_{2}=\frac{2\xi_{1}^{0}\left(  U_{y}b_{2}+U_{z}b_{3}\right)  \left(
\xi_{2}^{0}-U_{y}\xi_{1}^{0}t\right)  }{\left\vert \vec{\xi}\left(  t\right)
\right\vert ^{2}}, \label{eqn-b-2}%
\end{equation}%
\begin{equation}
\dot{b}_{3}=\frac{2\xi_{1}^{0}\left(  U_{y}b_{2}+U_{z}b_{3}\right)  \left(
\xi_{3}^{0}-U_{z}\xi_{1}^{0}t\right)  }{\left\vert \vec{\xi}\left(  t\right)
\right\vert ^{2}}. \label{eqn-b-3}%
\end{equation}
To show that $\Lambda_{m}=0$, it suffices to prove that $\left\vert \vec
{b}\left(  t\right)  \right\vert $ only has polynomial growth, uniformly in
$\left(  \vec{x}_{0},\vec{\xi}_{0},\vec{b}_{0}\right)  $. From equations
(\ref{eqn-b-2}) and (\ref{eqn-b-3}), we have
\begin{align*}
&  \frac{d}{dt}\left(  U_{y}b_{2}+U_{z}b_{3}\right) \\
&  =\frac{2\left(  U_{y}b_{2}+U_{z}b_{3}\right)  \left[  \xi_{1}^{0}\xi
_{2}^{0}U_{y}+\xi_{1}^{0}\xi_{3}^{0}U_{z}-\left(  \xi_{1}^{0}\right)
^{2}\left(  U_{y}^{2}+U_{z}^{2}\right)  t\right]  }{\left\vert \vec{\xi
}\left(  t\right)  \right\vert ^{2}}\\
&  =-\frac{\left(  U_{y}b_{2}+U_{z}b_{3}\right)  }{\left\vert \vec{\xi}\left(
t\right)  \right\vert ^{2}}\frac{d}{dt}\left\vert \vec{\xi}\left(  t\right)
\right\vert ^{2},
\end{align*}
by noting that
\begin{align*}
\left\vert \vec{\xi}\left(  t\right)  \right\vert ^{2}  &  =\left(  \xi
_{1}^{0}\right)  ^{2}+\left(  -U_{y}\xi_{1}^{0}t+\xi_{2}^{0}\right)
^{2}+\left(  -U_{z}\xi_{1}^{0}t+\xi_{3}^{0}\right)  ^{2}\\
&  =1-2\left(  \xi_{1}^{0}\xi_{2}^{0}U_{y}+\xi_{1}^{0}\xi_{3}^{0}U_{z}\right)
t+\left(  \xi_{1}^{0}\right)  ^{2}\left(  U_{y}^{2}+U_{z}^{2}\right)  t^{2}.
\end{align*}
Thus
\[
\frac{d}{dt}\left[  \left(  U_{y}b_{2}+U_{z}b_{3}\right)  \left\vert \vec{\xi
}\left(  t\right)  \right\vert ^{2}\right]  =0
\]
and
\begin{equation}
\left(  U_{y}b_{2}+U_{z}b_{3}\right)  \left(  t\right)  =\frac{U_{y}b_{2}%
^{0}+U_{z}b_{3}^{0}}{\left\vert \vec{\xi}\left(  t\right)  \right\vert ^{2}}.
\label{identity-combination}%
\end{equation}
For any fixed $t>0$, we find a lower bound for $\left\vert \vec{\xi}\left(
t\right)  \right\vert ^{2}$ by minimizing the function
\[
f\left(  \xi_{1}^{0},\xi_{2}^{0},\xi_{3}^{0}\right)  =\left(  \xi_{1}%
^{0}\right)  ^{2}+\left(  -U_{y}\xi_{1}^{0}t+\xi_{2}^{0}\right)  ^{2}+\left(
-U_{z}\xi_{1}^{0}t+\xi_{3}^{0}\right)  ^{2}%
\]
subject to the constraint $\left(  \xi_{1}^{0}\right)  ^{2}+\left(  \xi
_{2}^{0}\right)  ^{2}+\left(  \xi_{3}^{0}\right)  ^{2}=1$. By calculations of
Lagrange multiplier, we get
\begin{align*}
\min_{\left\vert \vec{\xi}_{0}\right\vert =1}\left\vert \vec{\xi}\left(
t\right)  \right\vert ^{2}  &  =\frac{2+\left(  U_{y}^{2}+U_{z}^{2}\right)
t^{2}-\sqrt{\left(  2+\left(  U_{y}^{2}+U_{z}^{2}\right)  t^{2}\right)
^{2}-4}}{2}\\
&  =\frac{2}{2+\left(  U_{y}^{2}+U_{z}^{2}\right)  t^{2}+\sqrt{\left(
2+\left(  U_{y}^{2}+U_{z}^{2}\right)  t^{2}\right)  ^{2}-4}}%
\end{align*}
So for $t>0$, we get the estimate
\[
\left\vert \vec{\xi}\left(  t\right)  \right\vert ^{2}\geq\frac{1}{2+\left(
U_{y}^{2}+U_{z}^{2}\right)  t^{2}}.
\]
Thus by (\ref{identity-combination}),
\[
\left\vert \left(  U_{y}b_{2}+U_{z}b_{3}\right)  \left(  t\right)  \right\vert
\leq c_{1}t^{2}+c_{2},
\]
for $c_{1},c_{2}>0$ independent of $\left(  \vec{x}_{0},\vec{\xi}_{0},\vec
{b}_{0}\right)  $. By (\ref{eqn-b-1})-(\ref{eqn-b-3}), we have
\[
\left\vert \dot{b}_{1}\right\vert \leq2\left\vert \left(  U_{y}b_{2}%
+U_{z}b_{3}\right)  \left(  t\right)  \right\vert ,\ \ \left\vert \dot{b}%
_{2}\right\vert ,\left\vert \dot{b}_{3}\right\vert \leq\left\vert \left(
U_{y}b_{2}+U_{z}b_{3}\right)  \left(  t\right)  \right\vert ,
\]
and thus
\[
\left\vert \vec{b}\left(  t\right)  \right\vert \leq c_{3}t^{3}+c_{4}%
\]
for some constants $c_{3},c_{4}>0$. This finishes the proof of the lemma.
\end{proof}

\end{document}